\documentclass[13pt]{amsart}

\usepackage{enumerate,url,amssymb,  mathrsfs, graphicx, pdfsync}
\newtheorem{theorem}{Theorem}[section]
\newtheorem{lemma}[theorem]{Lemma}
\newtheorem*{lemma*}{Lemma}

\newtheorem{proposition}[theorem]{Proposition}
\newtheorem{corollary}[theorem]{Corollary}
\theoremstyle{definition}
\newtheorem{definition}[theorem]{Definition}

\newtheorem{problem}[theorem]{Problem}

\theoremstyle{remark}
\newtheorem{remark}[theorem]{Remark}

\numberwithin{equation}{section}

\newcommand{\abs}[1]{\lvert#1\rvert}

\newcommand{\A}{\mathbb{A}}

\newcommand{\C}{\mathbb{C}}

\newcommand{\W}{\mathscr{W}}

\newcommand{\DD}{\mathbb{D}}
\newcommand{\IK}{\mathbb{K}}

\newcommand{\R}{\mathbb{R}}
\newcommand{\X}{\mathbb{X}}

\newcommand{\Y}{\mathbb{Y}}

\newcommand{\dtext}{\textnormal d}

\newcommand{\onto}{\xrightarrow[]{{}_{\!\!\textnormal{onto\,\,}\!\!}}}
\newcommand{\into}{\xrightarrow[]{{}_{\!\!\textnormal{into\,\,}\!\!}}}

\newcommand{\bydef}{\stackrel{\textnormal{def}}{=\!\!=}}

\DeclareMathOperator{\im}{Im}

\def\le{\leqslant}
\def\ge{\geqslant}

\begin{document}

\title{Energy-minimal  Principles in Geometric Function Theory.}
\maketitle

\begin{center}\textbf{Tadeusz Iwaniec,  Gaven Martin, Jani Onninen}\end{center}

\begin{center}In  memory of Sir Vaughan Frederick Randal Jones \end{center}

\begin{abstract}  

We survey a number of recent developments in geometric analysis as they pertain to the calculus of variations and extremal problems in geometric function theory following the NZMRI lectures given by the first author at those workshops in Napier in 1998 and 2005.
 
 \end{abstract}

\section{Introduction}

This article is a reflection of the authors' research program to advance variational techniques in Geometric Function Theory with many and varied applications in mind.  On one hand to things like nonlinear materials science, the calculus of variations, nonlinear analysis and PDEs,  and on the other to Teichm\"uller theory and Riemann surfaces. While making a few new observations,  the material is largely expository in nature and presents a number of recent advances from the NZMRI lectures given by the first author at those workshops in Napier in 1998 (Geometric Analysis) and 2005 (Interactions between Geometry and Algebra).  Indeed it was during that second conference that we (with Kari Astala - also in attendence) began \cite{AIMO} which set up the main initial problems in the area,  identified the connections between minimisers of mean distortion and harmonic mappings and the Nitsche phenomenon,  \cite{AIM},  which ultimately lead to the resolution of the Nitsche Conjecture from 1962, \cite{IKOni}.  Generalisations were discussed in \cite{IMO2}.

\medskip

The modern theory seeks an in-depth analysis of deformations which minimise naturally occurring energy functionals in geometric analysis and solves the associated Euler-Largrange (and other related) equations.  Crucially, we confine ourselves not just to examples,  but to use them to gain insights and new points of view to uncover new phenomena -- such as the Nitsche phenomena,  which unexpectedly identified conformal invariants as obstructions to the existence of nice minimisers for even the Dirichlet energy, see (\ref{1}) below.  Thus the main objects of our discussion are mappings
 \begin{equation}\label{TheMappings}\, h : \mathbb X \onto \mathbb Y\, \end{equation}
\noindent between given spaces of the same topological type; the map $h$ is typically, an orientation preserving homeomorphism and referred to as elastic deformation in materials science. Although $\,\mathbb X\,$ and $\,\mathbb Y\,$ could be general Riemannian $\,n\,$-manifolds (with or without boundaries),  we shall largely focus on the case they are Euclidean domains $\, \mathbb X , \mathbb Y \subset \,\mathbb R^n\,$. Our standing assumption is that $\,h\,$ belongs to some Sobolev space  $\,\mathscr W^{1,p}_{\textnormal{loc}}(\mathbb X, \mathbb R^n)\,$; where normally $\,1 <  p < \infty\,$. The case $\,p = 1\,$ will often  be treated as marginal due to the lack of reflexibility of  $\,\mathscr W^{1,1}(\mathbb X, \mathbb R^n)\,$; reflexibility is a fundamental prerequisite for the existence of energy-minimal mappings (usually called \textit{hyperelastic deformations})  via existence approaches based in the direct method of the calculus of variations.  Typically a studied problem will require a'priori bounds in a natural Sobolev space.  For instance in the classical case of determining the hyperelastic deformation of the energy functional
\begin{equation}\label{1} h \mapsto \int_{\mathbb X} | Dh(x) |^2 d\sigma_{\X}(x) \end{equation}
subject to any constraints,  one must seek minimisers in $\,\mathscr W^{1,2}_{\textnormal{}}(\mathbb X,\mathbb Y)$.  Here 
\begin{equation}
| Dh(x) |^2 = \sum_{i,j} \left| \frac{\partial h^j}{\partial x_k} \right|^2,  \quad x=(x_1,\ldots,x_N)\in\X, \quad h(x) = (h_1(x),\ldots,h_N(x)) \in \Y
\end{equation}
is the Hilbert-Schmidt norm.
\bigskip

Sobolev mappings between Riemannian manifolds can be defined in several ways, that are not necessary equivalent. We may, and do, assume that $\,\mathbb Y \,$ is a subset (not necessarily a subdomain) of $\,\mathbb R^N\,$ for some sufficiently large dimension $\,N\,$.  This follows from the celebrated embedding theorem of J. Nash~\cite{Nash}.  This being so, we say that $\,h :  \mathbb X  \rightarrow  \mathbb Y \subset \mathbb R^N\,$ belongs to $\,\mathscr W^{1,p}(\mathbb X , \mathbb Y)\,$ if each of its $\,N\,$ coordinate functions lies in the linear  space $\,\mathscr W^{1, p } ( \mathbb X, \mathbb R )\,$ and $\,h(x)  \in \mathbb Y\,$ for every  $\,x \in \mathbb X\,$; here there is a standard way of defining Sobolev scalar functions on a manifold. However, caution must be exercised because the topology of the target space $\,\mathbb Y\,$ (later referred to as the \textit{deformed configuration}) may prevent smooth mappings from being dense in  $\,\mathscr W^{1,p}(\mathbb X , \mathbb Y)\,$, \cite{HL2}.  

The following problem, still open,  gives a glimpse of the difficulties arising already at the very basic stages concerning Sobolev homeomorphisms.
\begin{problem}
Let $\mathbb Y \subset \R^n$ be a bounded topological ball (so homeomorphic  to the unit ball $\,\mathbb B\,$). For what $p$ (any ?) does there exist a homeomorphism $h \colon \mathbb B \onto \mathbb Y$ of Sobolev class $\W^{1,p} (\mathbb B , \Y)$ ?
\end{problem}

\section{Sobolev homeomorphisms as elastic deformations}   Elastic deformations of material bodies have fascinated famous scientists for centuries as one can see in the writings of,  for instance,  G. Galileo. H. Hook, L. Euler, J-L. Lagrange. T. Young. A. Cauchy, G. Green, G.R. Kirchoff.  Nowadays, mathematical models for nonlinear elasticity is an very active science. Here we will not cover all the major developments of the modern theory, but will discuss the  principles  with a determination  to rework, using other recent mathematical advances, some of the problematic presuppositions.  At the heart of these presuppositions is the \textit{Principle of Noninterpenetration of Matter},  \cite{Ba0}. This roughly asserts that an energy minimiser should always be a homeomorphism (and perhaps even a diffeomorphism depending on the problem).  Here we will see that one must really adopt the more general class of monotone  ($\,n= 2 \,$)  Sobolev mappings as legitimate \textit{hyperelastic} deformations of  elastic bodies. This leads to the ``\textit{Weak Principle of Noninterpenetration of Matter}'' and even in this apparently weaker context new topological/geometric arguments become even more essential.

\subsection{Hyperelasticity}
  In the mathematical models of nonlinear elasticity, we  study homeomorphisms $\,h \, : \mathbb X \onto \mathbb Y \,$ of smallest \textit{stored energy}, 
\begin{equation}\label{StoredEnergy}\,\mathscr E[h] = \int_\mathbb X \,\mathbf E(x, h, Dh)\;dx\;, \;\;\quad  \mathbf E : \mathbb X \times \mathbb Y \times \mathbb R^{n\times n} \rightarrow \mathbb R\,\end{equation}
where   the so-called \textit{stored energy function} $\,\mathbf E\,$ characterizes the various mechanical and elastic properties of the materials occupying the domains $\,\mathbb X\,$ and $\,\mathbb Y\,$.  We have also written $dx$ as a shorthand for whatever measure $d\sigma_\X(x)$ given as data on $\X$.
The \textit{ $\,p\,$-harmonic energy}, 
 \begin{equation}\label{HyperElasticEnergy}
 \mathscr E_p[h] = \int_\mathbb X |Dh(x)|^p \,dx \end{equation}
 including the Dirichlet integral ($\,p=2\,$),  is the ideal example to illustrate the scheme for the direct method in the Calculus of Variations,  and we will do so later after we introduce a few more natural stored energy functionals.

When $p\geq 1$ the integrand in (\ref{HyperElasticEnergy}) is a convex function of the differential $Dh$.  We will soon see why this is important.   The case $\, p = n\,$ will hereafter  be referred to as the \textit{Conformal Energy}.  This is  due to the fact that $\,\mathscr E_n[h] = \int_\mathbb X |Dh(x)|^n \,dx\,$ is invariant under a conformal change of the $\,x\,$-variable. 
 
 Let us also introduce two other natural functionals. First the so-called \textit{bi-conformal energy} defined for homeomorphisms $\,h \, : \mathbb X \onto \mathbb Y \,$ of Sobolev class $\,\mathscr W^{1,n}(\mathbb X , \mathbb Y)\,$ whose inverse $\,f \bydef h^{-1} : \mathbb Y \onto \mathbb X\,$  also lies in the Sobolev space $\,\mathscr W^{1,n}(\mathbb Y , \mathbb X)\,$

 \begin{equation}\label{BiconformalEnergy1}
 \mathscr E_n[h, f] = \int_\mathbb X |Dh(x)|^n \,dx\,\,\,+\;\;  \int_\mathbb Y |Df(y)|^n \,dy\,\,\,\end{equation}

  This can be formulated, equivalently via the change of variable $\, y = h(x)\,$ in the second integral,  by means of one \textit{polyconvex energy functional} for $\,h\,$ on $\,\mathbb X\,$:
 \begin{equation} \label{BiConformalEnergy2}
 \, \mathscr T[h] = \int_\mathbb X \left(\,|Dh(x)|^n \, +\,  |(Dh(x))^{-1}|^n J(x, h) \,\right) dx   
 \end{equation}
Here $J(x,h)$ is the Jacobian determinant of $Dh$ and polyconvexity, discussed further below, refers to the fact that the integrand is again a convex function of the minors of the differential $Dh$.  We don't assert that it is obvious that $ |(Dh(x))^{-1}|^n J(x, h)$ , or ${|Dh(x)|^n}/{ J(x, h)}$ are such a convex functions.  From the latter we consider the $q$-conformal energy functional
 \begin{equation} \label{qConformal}
 \, \mathscr K_q[h] = \int_\mathbb X \left( \frac{|Dh(x)|^n}{ J(x, h)} \,\right)^q dx   
 \end{equation}
 Here the integrand is actually a distortion function usually denoted
\begin{equation}\label{Kdef} \IK(x,h) = \frac{|Dh(x)|^n}{ J(x, h)} \end{equation}
 and is an infinitesimal measure of a the anisotropic nature of the deformation.  To see this note that for suitably regular deformations $h$ we can order the eigenvalues of $Dh^tDh$ as $0<\lambda_1^2\leq \lambda_2^2< \cdots<\lambda_n^2 <\infty$.  Then  
 
\[ n^{n/2} \left( \frac{ \lambda_1}{\lambda_n} \right)^n \leq  \IK(x,h)^2 = \frac{(\lambda_1^2+\lambda_2^2+\cdots+\lambda_n^2)^{n/2}}{\lambda_1.\lambda_2.\cdots.\lambda_n} \leq n^{n/2} \left( \frac{ \lambda_n}{\lambda_1} \right)^n\]
 Thus $\IK$ is controlled by the {\em linear distortion}, 
 \[ K=K(x,h)=\frac{\lambda_n}{\lambda_1}  = \limsup_{r\to 0} \frac{\max_{|\zeta|=r} |h(x+\zeta)-h(x)|}{\min_{|\zeta|=r} |h(x+\zeta)-h(x)|} \]
 In two dimensions actually $\IK(x,h)=\frac{1}{2}(K+1/K)$.
 A $\,\mathscr W^{1,n}(\mathbb X , \mathbb Y)\,$ homeomorphism between Euclidean domains with $\IK\in L^\infty(\X)$ (equivalently $K\in L^{\infty}(\X)$) is called {\em quasiconformal}. Then $\mathscr K_q[h]$ is sometimes referred to as the $q$-mean distortion.  

 \bigskip

To understand  the existence problem for hyperelastic deformations,  that is minimisers, we must first accept the weak limits of energy-minimizing sequences of homeomorphisms as legitimate deformations. Thus we allow for \textit{weak interpenetration of matter}; roughly speaking, squeezing of portion of the material to a point can occur, but not folding or tearing. This potentially changes the nature of a minimisation problem to the extent that the minimal energy among such weaker deformations might be strictly smaller than the infimum energy among homeomorphisms. Indeed this can happen. However when it does not,  one might subsequently attempt to explain why this squeezing doesn't happen.  A classical example of this topological regularity is  the Rado-Kneser-Choquet theorem for harmonic mappings from 1926 \& 1945,  \cite{Ch,Kn} \\

Different kinds of variational problems  occur naturally in geometric function theory. The first and most important example of course is the Riemann mapping theorem which we now consider in the above context. 

\subsection{Conformal Mappings are frictionless-minimizers of Dirichlet Energy}  

Frictionless refers to problems where we do not prescribe what a minimiser to to be on the boundary of $\X$.  We are merely given $\X$ and $\Y$ as configurations. 

\begin{theorem}\label{RiemannMapping} Let $\,h : \mathbb X \onto \mathbb Y\,$ be a conformal map  between bounded domains
$\,\mathbb X , \mathbb Y \subset \mathbb R^2 \simeq \mathbb C\,$.  Then every orientation preserving homeomorphism $\, f : \mathbb X \onto \mathbb Y\,$ of Sobolev class $\,\mathscr W^{1,2}(\mathbb X , \mathbb C)\,$ has Dirichlet energy at least that of $\,h\,$.  Equality occurs if and only if $f$ is conformal as well.
\end{theorem}
 \begin{proof}
 Using complex variables $\, z = x_1 + i x_2 \in \mathbb X\,$,  $\frac{\partial}{\partial z}= \frac{1}{2}\Big( \frac{\partial}{\partial x}-i\frac{\partial}{\partial y}\Big)$ and so forth,  the statement  reads as follows
\begin{eqnarray*}
\mathscr E_2[h] & = &  2 \int_{\mathbb X} |h_{z}|^2 + |h_{\overline{z}}|^2  =  2 \int_{\mathbb X} |h_{z}|^2 - |h_{\overline{z}}|^2  
= 2 \int_{\mathbb X} J(z,h) dz\\& = &  2 |\mathbb Y| =   2 \int_{\mathbb X} J(z,f) dz = 2 \int_{\mathbb X} | f_{z}|^2 - | f_{\overline{z}}|^2 \leqslant
 \mathscr E_2[f]
 \end{eqnarray*}
 Equality occurs if and only if $f_{\overline{z}} = 0 $, which from Weyl's lemma implies that $\,  f : \mathbb X \onto \mathbb Y\,$ is conformal.
 \end{proof}
 Here we have used the fact that both Jacobians  $\, J(z,h) = |h_{z}|^2 - |h_{\overline{z}}|^2 \,$ and
$\,J(z,f) =  |f_{z}|^2 - | f_{\overline{z}}|^2\,$ have the same integral over $\,\mathbb X\,$, namely  the area of $\,\mathbb Y\,$. This property is given  to more general nonlinear differential expressions called {\em Free Lagrangians} we discuss later. 

 \begin{remark}
 As remarked above, conformal mappings are solutions to the \textit{first order differential equations} $f_{\bar z}=0$,  originally introduced by D'Alembert and   traditionally referred to as  the Cauchy-Riemann equations. By contrast, the variational equations for the stored energy at (\ref{StoredEnergy})
are second order PDEs (in this case the Laplacian) which is usually subject to prescribed boundary values  $\,h_0 : \partial \mathbb X \rightarrow \partial \mathbb Y\,$. Typically in the Riemann mapping problem we are only given the initial domain $\mathbb{X}=\mathbb{D}$ and the target domain $\,\mathbb Y\,$ (that is  the {\em shape} of the deformed configuration) without specifying the boundary values of $\,h\,$ - that is how the boundary should be deformed. Of course $\mathbb{X}$ may also be any simply connected domain.

With boundary restrictions the minimisation problem would be \textit{ill posed}. That is one cannot prescribe the boundary values of a conformal mappings.  The minimisation problem would yield a solution within harmonic deformations,  but whose real and imaginary parts need not be harmonic conjugate.  

 In fact, this is the  simplest and most natural example of a general frictionless problems in the calculus of variations. frictionless problems concern energy-minimal deformations $\,h _\circ :\mathbb X \onto \mathbb Y\,$ (usually homeomorphisms) with no prescribed boundary map  $\,h_\circ:  \partial \mathbb X \onto \partial \mathbb Y\,$; in other words, tangential slipping along $\,\partial \mathbb X\,$ is allowed. In nonlinear elasticity this is physically realised when deforming confined incompressible material. The use of the direct method for these sorts of problems has brought us to the concept of \textit{Free Lagrangians}.  However first, let us briefly outline the evolving concept of \textit{Null Lagrangians}, extensively discussed  and developed in the celebrated
paper \cite{BallCurrieOlver}. This paper includes the references  for earlier significant earlier contributions to this idea.  \end{remark}

\section{Direct Method for $\,p\,$-harmonic energy.} \label{pHarmonicExample} As above, a
 representative  example  for convex functionals is  the $\,p\,$-harmonic energy of mappings  $\,h : \mathbb X \rightarrow \,\mathbb R^m\,$ with prescribed boundary conditions.

 \begin{equation}\label{ConvexEnergy}\,\mathscr E_p[h] \; =  \int_\mathbb X \,|\,Dh(x)\,|^p \,\textnormal{d} x\; ,\;\;\;\; h \in h_\circ  + \,\mathscr W^{1,p}_0 (\mathbb X, \mathbb R^m)\;\;,\;  1<p <\infty\;.     \end{equation}
Here,  the given mapping $\,h_\circ \in \mathscr W^{1,p}(\mathbb X, \mathbb R^m)\,$ takes the role of  boundary data in the   \textit{weak formulation of the Dirichlet problem}. The function space $\mathscr W^{1,p}_0$ refers to those functions vanishing on the boundary (in the Sobolev sense - the closure of $\mathscr{C}^{\infty}_{0}(\X)$ in the appropriate norm).

 \bigskip

We note the following aspects of the setup here.
\begin{itemize}
\item  We are naturally using the separable reflexive Banach space
$$\,\mathfrak B = \mathscr W^{1,p}(\mathbb X, \mathbb R^m)\,$$ 
\item  The functional $\,\mathscr E_p\,$, subject to minimization,  is defined on a subset
$$\,\mathfrak B_\circ \bydef h_\circ  + \,\mathscr W^{1,p}_0 (\mathbb X, \mathbb R^m)\,$$
 which is closed in the weak topology of $\,\mathfrak B\,$.
\item  {\em Coercivity Condition}. We have a condition which controls the $\,\mathfrak B\,$-norm of  $\, f \in \mathfrak B_\circ \,$ by means of its energy $\,\mathscr E_p[f]\,$. In the above example, we have a routine estimate
\begin{equation}\label{pHarmonicCoercivity}\,
   |\!| f |\!|_{_{\mathfrak B}}^p \;\bydef  \int_{\mathbb X }|Df|^p  + |\,f\,|^p  \; \sim \;  |\!| h_\circ |\!|_{_{\mathfrak B}}^p \;+\;
   \mathscr E_p[f]\;,\;\; \textnormal {for every}\,\,f \in \mathfrak B\,
 \end{equation}
 \item \textit{An energy-minimizing sequence} has a limit.  That $\,h_\kappa \in \mathfrak B_\circ\,, \; \kappa = 1,2,... \;$ ,  means that $$\,\inf\{\mathscr E_p[h] \,;\, h \in \mathfrak B_\circ   \}= \lim \mathscr E_p[h_\kappa] .\,$$
       Since $\,\mathfrak B\,$ is reflexive, we may extract from $\,\{h_\kappa\}\,$ a subsequence, still denoted by $\,\{h_\kappa\}\,$, converging weakly to a mapping $\,h_\infty \in \mathfrak B_\circ\,$.
  \item \textit{Lower semicontinuity}.  Now, everything hinges on establishing the inequality
  \begin{equation}\label{Semicontinuity}
  \mathscr E[f] \; \leqslant \liminf_{i \rightarrow \infty} \,\mathscr E[f_i]\;,\; \textnormal{whenever}\; \,f_i  \rightharpoonup f\, , \,\textnormal{ weakly in}\;  \,\mathfrak B\,
  \end{equation}
Here we are only interested in (\ref{Semicontinuity}) for energy minimising sequences,  but as a general property the sequence $\,\{f_i\}\,$ need  not be  energy-minimizing. Customary terminology refers to the energy functional satisfying  (\ref{Semicontinuity})    as being  \textit{sequentially weakly lower semicontinuous}. For simplicity we omit the words``sequentially weakly". The route to establishing lower semicontinuity goes through a subgradient estimate.

\item \textit{Subgradient Estimate}. For an integrand $\,\mathbf E(x,y, \xi)\,$ that is convex with respect to $\, \xi \in \mathbb R^{m\times n}\,$  we have a subgradient estimate:
    \begin{equation}
    \mathbf E(x, y, \xi) \; - \mathbf E(x,y, \xi_\circ) \;  \geqslant \; \Big{\langle}\,\nabla_\xi\mathbf E(x,y, \xi_\circ)\; \;\Big{|}\;\; \xi - \xi_\circ \; \Big{\rangle}
    \end{equation}
    Here the symbol $\, \langle\; | \; \rangle\,$ stands for the inner product of $\,m\times n\,$-matrices .  In our $\,p\,$-harmonic example this inequality, upon integration,  reads as
    \begin{equation}
     \int_\mathbb X |Df(x)|^p\, dx \; -\;  \int_\mathbb X |Df_\circ(x)|^p\, dx \;\geqslant p \int_\mathbb X \Big{\langle}\,|Df_\circ |^{p-2} Df_\circ\; \;\Big{|}\;\; Df\,-\, Df_\circ \; \Big{\rangle}
    \end{equation}
    \end{itemize}
    
Now we take $\, f = h_\kappa \, $ and $\,f_\circ = h_\infty\,$.  Then $\, Dh_\kappa  - Dh_\infty\, $ converges weakly to zero in the space $\,\mathscr L^{\,p}(\mathbb X, \mathbb R^{m\times n})\,$  and $\,|Df_\circ |^{p-2} Df_\circ \,$ lies in the dual space $\,\mathscr L^{\,q}(\mathbb X, \mathbb R^{m\times n})\,,\; \frac{1}{p}  + \frac{1}{q} = 1\,$. Hence

    $$
     \int_\mathbb X |Dh_\kappa(x)|^p\, dx \; \, \rightarrow \;\; \int_\mathbb X |Dh_\infty (x)|^p\, dx 
    $$
    as required. 
   
   \bigskip

    \begin{remark} This natural illustration of the direct method actually has wide-ranging enhancements in which the role of  weak convergence $\, Dh_\kappa \rightharpoonup Dh_\infty\,$ is taken by \textit{Null Lagrangians}.  We discuss these next.

    \end{remark}

 \section{Null Lagrangians and polyconvex functionals}

\subsection{Null Lagrangians} Consider a Sobolev mapping $\,f = (f^1,... , f^m) :\mathbb X \rightarrow \mathbb R^m\,$  of an open region $\,\mathbb X \subset \mathbb R^n\,$ into $\,\mathbb R^m$.
 The term \textit{null Lagrangian} pertains to
 a nonlinear differential $\,n\,$-form $\,\mathbf N(x,f, Df)\,dx\,$, whose integral mean over
 any open subregion $\, \Omega \subset \mathbb X\,$ depends only on the boundary values of $\,f : \partial \Omega \rightarrow \mathbb R^m\,$.  While there are technicalities here such as defining a Sobolev function on $\partial \Omega$,  typically these things are not at issue when $f$ is at least continuous.

 This condition is reminiscent of that for exact differential forms, by virtue of Stokes' formula. It may very well be right to call such
expressions  \textit{nonlinear exact  forms}.  Distinctive examples are furnished by the
\emph{subdeterminants of the $\,m \times n\,$ matrix} of the linear tangent map $\, Df : \mathbb R^n \into \mathbb R^m\,$  called the \textit{deformation gradient},

$$
 Df(x) \,\bydef\,  \left[
 \begin{array}{cccc}
\frac{\partial f^1}{\partial x_1}&\frac{\partial f^1}{\partial
x_2}&

\ldots&\frac{\partial f^1}{\partial x_n}\\
\\

\frac{\partial f^2}{\partial x_1}&\frac{\partial f^2}{\partial x_2}&
\ldots&\frac{\partial f^2}{\partial x_n}\\

\vdots&\vdots&{}&\vdots\\

\frac{\partial f^m}{\partial x_1}&\frac{\partial f^m}{\partial
x_2}&\ldots&\frac{\partial f^m}{\partial x_n}
\end{array}\right]\, = \left[\frac{\partial f^i}{\partial
x_j}\right]\in\mathbb{R}^{m\times n},\;i=1,...,m\,,\,j=1,...,n
$$
To every pair $(I,J)$
of ordered $\ell$-tuples $I:\; 1\leqslant
i_1<i_2<...<i_\ell\leqslant m$ and $J:\; 1\leqslant
j_1<j_2<...<j_\ell\leqslant n$, with   $1\leq \ell\leqslant
\min\{m,n\}$, there corresponds an  $\ell\times \ell$
-minor \index{minor} of $ Df(x)$, denoted by

\begin{equation}\label{8}
  \frac{\partial f^I}{\partial\, x_{_{\!J}}}  = \frac{\partial
(f^{i_1},...\,,f^{i_\ell})}{\partial (x_{j_1},...\,,x_{j_\ell})}
\end{equation}\\
These  minors,  are the
coefficients of the wedge product:
$$
     df^{i_1}\wedge ...\wedge df^{i_\ell} =\sum_{1\leqslant
     j_1<...<j_\ell\leqslant
      n} \frac {\partial (f^{i_1},...\,,\,f^{i_\ell})}{\partial
      (x_{\!j_1},...\,,\,x_{\!j_\ell})}\;\; dx_{\!j_1}\wedge ...\wedge
      dx_{\!j_\ell}=\text{\Large$\sum_J$}\;\frac{\partial f^I}{\partial x_{_{\!J}}}\; dx_{_{\!J}}
$$\\
Stokes' formula tells us 
that for  $ \,f,g \in \mathscr
W^{1,\,\ell}(\Omega, \mathbb R^m)$ we have
\begin{equation}
 \underset{\Omega}{\int} \frac {\partial f^I}{\partial x_{_{\!J}}}
\;\mathrm dx \;= \;\underset{\Omega}{\int} \frac {\partial
g^I}{\partial x_{_{\!J}}} \;\mathrm dx,\;\;\;\;\;\;\;\;\;\;\;\text{provided}\;\; f-g
\in \mathscr W^{1,\,\ell}_\circ(\Omega, \mathbb R^m)
\end{equation}

This leads to the affine combinations (with constant coefficients) of the Jacobian subdeterminants as examples of null Lagrangians
\begin{equation}\label{NullLag}
 \mathbf N(x,f, \,D\!f) \;= \overset{\; \text{min}\{m,n\}}
 {\underset{\ell\;=\;0}
{\text{\LARGE{$\sum$}}}}  \;\;\underset{  {1\leqslant
i_1<...<i_\ell\leqslant m  } \atop {1\leqslant
j_1<...<j_\ell\leqslant n} } {\text{\LARGE{$\sum$}}}\;\;\;\;
\text{\LARGE{$\lambda$}}_{i_1...i_\ell}^{j_1...j_\ell} \;\; \frac
{\partial
(f^{i_1},...\,,\,f^{i_\ell})}{\partial(x_{\!j_1},...\,,\,x_{\!j_\ell})} \;\bydef \; \mathbf N(D\!f)
\end{equation}\\
where we adhere to the convention that the term with
$\,\ell = 0\,$ is a constant function.  In fact we have the following characterisation.
\begin{theorem}\label{BasicNullLagrangians}
Formula (\ref{NullLag}) represents all null Lagrangians of the form $$\,\mathbf{N}(x, f, Df)\, d x \; =\; \mathbf {N}(Df)\, d x\,$$  
\end{theorem}
This result goes back to \cite{Landers,Edelen, Ericksen, Rund}.
Should it be required to appeal to a first order null Lagarangians of the general form $\,E(x,f,D\!f)\,dx$, we refer  to   the work by  de Franchis \cite{deFranchis}.

The utility of  null Lagrangians is best illustrated for polyconvex functionals discussed next.

\subsection{Polyconvexity}  In his mathematical models for nonlinear elasticity  \cite{Bac}  J. Ball made the crucial observation that if the convexity of the stored energy integrand $\,\mathbf E(x, h, Dh )\,$ , with respect to the deformation gradient $\,Dh(x) \in \mathbb R^{m\times n}\,$, must be ruled out, it could be replaced by a weaker requirement; namely, expressing the integrand as a convex function of  subdeterminants of $\,Dh\,$.
\begin{equation}\label{PolyconvexIntegrand}
\mathbf E(x,h, Dh) \,= \,  \mathbf E_{_\maltese}(x,h,\, \textnormal{subdeterminants of} \; Dh)
\end{equation}
The number of all $\,\ell\times \ell\,$-subdeterminants with  $\,0 \leqslant \ell \leqslant \min\{m , n\}\,$   is equal to $\, {m+n \choose n }  \,$. Thus we are assuming that for every pair  $\,(x, y) \in \mathbb X\times \mathbb Y\,$  the function
$$
\mathbf E_{_\maltese}(x,y,\, \cdots)\, ; \,\mathbb R^{^{{m+n \choose n }}} \,\rightarrow \mathbb R\;, 
$$
is convex.  The idea of minimizing polyconvex energy functionals is based on  a quite  far reaching extension of the direct method of the calculus of variations that we outlined earlier.  It has turned out that so far this is the only practical idea that offers substantially more than that of just minimising convex energies,  for instance   the $p$-harmonic example.

\begin{remark} There is an  extensive literature dealing with Morrey's notion of \textit{Quasiconvexity}, \cite{Morrey}. However, from the point of view of mathematical challenges, this concept  is not much more than  a reformulation of the lower semi-continuity of the energy functionals and as such there remain only technical issues. One needs to develop this idea much  further mathematically before it might be usefully applied.
\end{remark}

\subsection{Nearly Conformal Deformations} \label{NearlyConformal} Given a bounded domain $\,\mathbb X \subset \mathbb R^n\,$,
 we  look at the  mappings $\, h  : \mathbb X \rightarrow \mathbb R^n\,$ in the reflexive Banach space  $\,\mathfrak B \bydef \mathscr W^{1, np}(\mathbb X, \mathbb R^n )\,,\; 1 \leqslant p < \infty\,$. Then the following nonlinear functional is well defined on this space.

\begin{equation}\label{PolyConvexEnergy}\,\mathscr E_{n,p}[h] \; =  \int_\mathbb X \,\left(|\,Dh(x)\,|^n \;-\; n^{n/2}\, \textnormal{det} Dh(x) \right)^p \,\textnormal{d} x\;  \;\; < \infty    \end{equation}
 Note that the integrand is non-negative and vanishes only if $\,h\,$  satisfies the $\,n\,$-dimensional variant of the Cauchy-Riemann system.
 \begin{equation} \label{nDimensionaldAlembert}
 \mathcal K(Dh) \equiv 0\;,\; \;\textnormal{where}\; \;\mathcal K(X) \bydef |X|^n - n^{n/2}\, \textnormal{det}\,X \, \geqslant 0\;\;,\;\; \textnormal{for}\, X \in \mathbb R^{n\times n}
 \end{equation}
 This motivates our calling $\,\,\mathscr E_{n,p}\,$ a \textit{nearly conformal energy functional}.
 The Dirichlet boundary value problem consists of minimizing $\,\mathscr E_{n,p}[h]\,$ subject to mappings $\,h  \in  \mathfrak B_\circ \bydef
 h_\circ  + \,\mathscr W^{1,n p}_0 (\mathbb X, \mathbb R^n)\,$,  where $\,h_\circ \in \mathfrak B\,$ is given boundary data.
  Here are the essential  steps in the Direct Method. \\

 \begin{itemize}
  \item  \textit{\textbf{Coercivity in the mean}}   The reason why the above example is approachable is that while we do not have point-wise coercivity in terms of  the integrand,   the energy functional  $\,\mathscr E_{n,p}\,$ still exhibits  \textit{coercivity in the sense of integral means}; precisely,
     \begin{equation} \label{MeanCoercivity}
     \int_\mathbb X  |Dh(x)|^{np} dx \; \sim \; \mathscr E_{n,p} [h] \; +\; \int_\mathbb X  |Dh_\circ(x)|^{np} dx
      \end{equation}
      To see this we appeal to the  general (but not vary obvious) estimate (10.12) in \cite{IL-arma}  for mappings $\,f \in \mathscr W^{1, np}_\circ (\mathbb X, \mathbb R^n)\,$
  \[ 
 \int_\mathbb X  |Df(x) | ^{np} \,\textnormal{d} x  \sim \; \mathscr E_{n,p} [f]\;
\] 
applied to the mapping $f = h - h_\circ$.
     \item  \textit{\textbf{Subgradient Estimate}} This is  a fairly direct consequence of polyconvexity of the integrand. Precisely, for $\,n\times n\,$-matrices $\, X, X_\circ\in \mathbb R^{n\times n}\,$ we have
\begin{eqnarray*}
         \mathcal K^p(X) \,-\, \mathcal K^p(X_\circ) & \geqslant & p \,\mathcal K^{p-1}(X_\circ)\; \big {[}\, \mathcal K(X)\;-\; \mathcal K(X_\circ) \,\big{]} \\
       & \geqslant &  n \,p\,  \Big{\langle } \mathcal K^{p-1}(X_\circ)\; |X_\circ |^{n-2} X_\circ \,\Big {|} \,\,X \, -\, X_\circ \Big{\rangle }  \\
&& - n^{n/2} \,p\,\,   \mathcal K^{p-1}(X_\circ)\;\Big{[} \textnormal{det}\,X\;-\; \textnormal{det} \,X_\circ  \Big{]}
\end{eqnarray*}
  \item \textit{\textbf{Lower semi-continuity}}  The required inequality  (\ref{Semicontinuity})\, can be achieved by applying the above subgradient estimate  to $\,X = Df_i(x)\,$  and $\,X_\circ = Df(x)\,$. We conclude, upon integrating over $\,\mathbb X\,$, that
\begin{eqnarray*}
      \lefteqn{   \mathscr E_{n,p}[f_i]   - \mathscr E_{n,p}[f]  }\\& \geqslant & n \,p\, \int_\mathbb X \Big{\langle } \mathcal K^{p-1}(Df)\; |Df |^{n-2} Df \,\Big {|} \,\,Df_i\, -\, Df \Big{\rangle \;\; \;\;\;(\;\textnormal{converging to}\; \; 0\,) }\\
&& - n^{n/2} \,p\, \int_\mathbb X  \mathcal K^{p-1}(Df)\;\Big{[} \textnormal{det}\,Df_i\;-\; \textnormal{det} \,Df \Big{]} \;\; \;\;\;\;\;(\;\textnormal{converging  to} \;\; 0\,)
\end{eqnarray*}

      The first limit is justified by the fact that $\,Df_i \rightharpoonup Df\,$, weakly in the space $\,\mathscr L^{np}(\mathbb X)\,$,  and the integration takes place against the factor $\,\mathcal K^{p-1}(Df)\; |Df |^{n-2} Df \,\approx\, |Df|^{np - 1} \,$ which lies in the dual space $\,\mathscr L^{\frac{np}{np-1}}(\mathbb X)\,$.  Similarly, for $\,p > 1\,$,  the null Lagrangians  $\,\textnormal{det}\,Df_i \,$ converge to  $\textnormal{det}\,Df\,$ weakly in $\,\mathscr L^p(\mathbb X)\,$  and we integrate them against a function   $\,\mathcal K^{p-1}(Df)\approx | Df|^{np-n}  \,$ which belongs to the dual space $\,\mathscr L^{\frac{p}{p-1} }(\mathbb X)\,$. 
      
      The case $\,p=1\,$ needs  handling with greater care. It is not generally true that
      $\,\lim \int_\mathbb X  \textnormal{det}\,Df_i(x) \,dx \;=\; \int_\mathbb X  \textnormal{det}\,Df(x) \,dx \;$,\; whenever\,\; $\,f_i \rightharpoonup f\,$, weakly in\; $\,\mathscr W^{1,n}(\mathbb X)\,$.  For instance the sequence of M\"{o}bius transformations  $\, f_i : \mathbb B \onto \mathbb B\,$ of the unit ball $\,\mathbb B \subset \mathbb R^n\,$ such that $\, f_i(0) \rightarrow  a  \in\,\partial \mathbb B\,$. Their weak limit (indeed locally uniform in  $\mathbb B$) is $\,f \equiv a\,$,  which has vanishing Jacobian.  However  $\, \int_\mathbb B  \textnormal{det}\,Df_i(x) \,dx \,=\, |\mathbb B\,| > 0\,$. 
      
      The situation is quite different if we confine ourselves to the energy-minimising sequence of mappings  $\, f_i \in \mathfrak B_\circ\,$ in which $\, f_i \,\in\, f_0  +\;  \mathscr W^{1,n}_0 (\mathbb X)\,$. Once the boundary values of $\,f_i\,$ are fixed,  the weak limit enjoys the same boundary values; that is,  $\, f \,\in\, f_0  +\;  \mathscr W^{1,n}_0 (\mathbb X)\,$. Thus, for all $\,i = 1,2, ...\,$,  we have the following identities
      $$
      \int_\mathbb X \textnormal{det} \,Df_i(x)\, \textnormal dx \;=\;\int_\mathbb X \textnormal{det} \,Df_0(x)\, \textnormal dx \;=\;\int_\mathbb X \textnormal{det} \,Df(x)\, \textnormal dx\;,
      $$
      by the very definition of null Lagrangians.
        \end{itemize}
        
      Now,  with these estimates at hand we may follow the principles of the direct method,  and thereby obtain the following.  
      \begin{proposition}\label{ExistenceNearlyConformalDeformations}
      The nearly conformal energy (\ref{PolyConvexEnergy}) , subject to a given boundary data $\,h_\circ \in \mathscr W^{1,np}(\mathbb X, \mathbb R^n)\,$, attains its infimum.
      \end{proposition}

The example above shows how the over-arching  strategy of the direct method can be nuanced in sophistication at each step. The following specific example further demonstrates this point.

\section{A Very Weak Domain of Definition for  Nearly Conformal Deformations}  \label{WeaklyConformal} Given a bounded domain $\,\mathbb X \subset \mathbb R^n\,$, we look at the energy functional

\begin{equation}\label{VeryWeaklyConformalFunctional}
\mathscr E^\sharp [h]  \bydef \; \varepsilon \,\Big (\int_\mathbb X |Dh|^{n-1} \,\Big) ^{\frac{n}{n-1}}\; +\; \int_\mathbb X \Big[ \lambda\, |D^\sharp h|^{\frac{n}{n-1}} \; -\; n^{n/2n-2} \, \textnormal{det} Dh \Big]
\end{equation}
with the fixed constant parameters $\,\varepsilon >0\,$  and $\,\lambda > 1\,$. Here the notation $\,X^\sharp\,$ for $\,X \in \mathbb R^{n\times n}\,$ stands for the cofactor matrix $\,X^\sharp \in \mathbb R^{n\times n} \,$ whose entries are  $\,\pm\,$ subdeterminants of size $\,(n-1) \times (n-1)\,$ of the matrix  $\,X\,$. Thus the differential $\,n\,$-form $\,[D^\sharp h]\,dx\,$ is a matrix of null Lagrangians.  We always have  $\,  |X^\sharp |^{\frac{n}{n-1}} \; -\; n^{n/2n-2} \, \textnormal{det} X  \geqslant  0 \,$,  equality occurs if and only if $\,X\,$ is a similarity matrix of  non negative determinant. Thus we always have $\, \mathscr E^\sharp [h]  \geqslant 0\,$. 

The border line case of $\,\varepsilon = 0\,$ and $\,\lambda = 1\,$ reduces to the energy functional which distinguishes orientation preserving conformal mappings as its absolute minima; that is, $\,\mathscr E^\sharp [h] = 0\,$. When $n=2$ such minima are either constant mappings, M\"{o}bius transformations or a holomorphic function.

\medskip

From a different perspective energy-minimizers of these type of functionals are discussed in \cite {MullerQiYan}, where the existence results are presented  for mappings with $\,\textnormal{det} Dh(x) \geqslant 0\,$, see Lemma 4.1 therein. Let us demonstrate here how to remove this assumption.

\medskip

At first glance the space $\,\mathscr W^{1,n}(\mathbb X, \mathbb R^n) \,$ would seem to be the natural domain of definition of $\, \mathscr E^\sharp [h]\,$. Unfortunately,  energy-minimizing sequences need not be bounded in this space.  We recall the direct method in this particular case to make a few observations along the way.
\begin{itemize}
\item   \textit{\textbf{Seek energy-minimizers in the Banach space}}
\[ \mathfrak B \bydef \mathscr W^{1,n-1}(\mathbb X, \mathbb R^{n-1}) \]
\item \textit{\textbf{$\,\mathfrak B_\circ \subset \mathfrak B\,$}} This time the subset  $\mathfrak B_\circ$ is somewhat tricky to define and needs handling with care. We have the given boundary data  $\,h \in h_\circ +  \mathscr W^{1, n-1}_\circ(\mathbb X, \mathbb R^n)$. The matrix of cofactors $\,D^\sharp h_\circ \,$   must also lie in the space $\mathscr L^{\frac{n}{n-1}}(\mathbb X, \mathbb R^{n\times n})$  just to secure the  finite energy of $\,h_\circ\,$. Therefore, we define  $\,\mathfrak B_\circ \subset \mathfrak B\,$ by the following rule.
\[ \mathfrak B_\circ \bydef \{ h \in h_\circ +  \mathscr W^{1, n-1}_\circ(\mathbb X, \mathbb R^n)\;: \;\,  D^\sharp h \in \mathscr L^{\frac{n}{n-1}}(\mathbb X, \mathbb R^{n\times n})\;\}\]
    Hence $\, \mathscr E^\sharp [h] \,< \infty\,$ for every $\,h \in \mathfrak B_\circ\,$.
    \item \textit{\textbf{Coercivity}} The restriction to  mappings $\,h \in \mathfrak B_\circ\,$  also secures a coercivity estimate as follows.
\[ \Big (\int_\mathbb X |Dh(x)|^{n-1} dx\,\Big) ^{\frac{n}{n-1}}\; +\; \int_\mathbb X \Big( |D^\sharp h(x)|^{\frac{n}{n-1}} dx \;\Big)\; \;\sim \; \mathscr E^\sharp [h]\;< \infty
\]
     Here the implied constant in front of the energy $\, \mathscr E^\sharp [h]\;\,$ depends only on $\,\varepsilon > 0\,$ and $\,\lambda >1\,$; precisely, it can be shown that it does not exceed $\, \max\{\frac{1}{\varepsilon} \,\frac{1}{\lambda - 1}\}\,$
    \item \textit{\textbf{An energy-minimizing sequence.}} Denoted by   $\{h_\kappa\} \subset \mathfrak B_\circ\,$ and approaching the infimum energy, 
\[ \lim_{\kappa \rightarrow \infty} \, \mathscr E^\sharp[h_\kappa]\; = \;\inf\{\,\mathscr E^\sharp[h]\, \,;\,\; h \in \mathfrak B_\circ \} \bydef \textsf{E} \]

         Then $\,h_\kappa \,$ are bounded in the Sobolev space $\,\mathscr W^{1,n-1}(\mathbb X , \mathbb R^n )\,$  and $\,\{ D^\sharp h_\kappa\}\,$ are bounded in  $\,\mathscr L^{\frac{n}{n-1}}(\mathbb X, \mathbb R^{n \times n})\,$\,, both being reflexive Banach spaces. We extract and fix a subsequence, again denoted by $\,\{h_\kappa \}\,$, such that:
         \begin{itemize}
         \item
         $
        h_\kappa \rightharpoonup h_\infty \;\;\textnormal{weakly in\,} \mathscr W^{1,n-1}(\mathbb X , \mathbb R^n )\,$
        \item  \;
        $\,D^\sharp h_\kappa\,$ converges weakly in $\, \mathscr L^{\frac{n}{n-1}}(\mathbb X, \mathbb R^{n \times n}) $ to a matrix field, say $\,  \mathcal M  \in  \mathscr L^{\frac{n}{n-1}}(\mathbb X, \mathbb R^{n \times n}) $.
        \item Note that  $\,D^\sharp h_\kappa\; \in \mathscr L^1(\mathbb X ,\mathbb R^{n\times n})\,$ converge to $\,D^\sharp h_\infty\,$ in the sense of distributions. That is
\[ \int_\mathbb X \eta(x)\,D^\sharp h_\kappa(x)\;\textnormal dx \;\rightarrow \int_\mathbb X \eta(x)\,D^\sharp h_\infty(x)\;\textnormal dx\; ,\,\textnormal{for every}\; \eta \in \mathscr C^\infty_0(\mathbb X).
\] 
        This property is referred to in the literature as the  \textit{weak continuity of Jacobians}. Its discovery goes back at least as far as the forgotten paper by R. Caccioppoli \cite{ForgotenCaccioppoli}.  Hence, it is easily seen that $\,\mathcal M(x) = D^\sharp h_\infty(x)\,$,  almost everywhere.
        \end{itemize}

  \item \textit{\textbf{Lower semicontinuity}}.  Now a  real challenge emerges from trying to evaluate the limit of $\,\int_\mathbb X \textnormal{det} Dh_\kappa(x)\;\textnormal dx \,$. Although the Jacobians $\,\textnormal{det} Dh_\kappa\,$ remain bounded in $\,\mathscr L^1(\mathbb X)\,$ it is not generally guaranteed that the integrals $\,\int_\mathbb X \textnormal{det} Dh_\kappa \,$ admit a  subsequence converging  to  $\,\int_\mathbb X \textnormal{det} Dh_\infty \,$. A difficulty is to be expected because the gradients $\,\{Dh_\kappa \}\,$ need not be  bounded in $\,\mathscr L^n(\mathbb X, \mathbb R^{n\times n})\,$ (when $\,n > 2\,$).

       \begin{remark}  At this point it is  worth recalling the concept of \textit{distributional Jacobians},   defined for mappings $\, f = (f^1, f^2, ... , f^n)\,$ of Sobolev class $\,\mathscr W^{1, n-1}(\mathbb X, \mathbb R^n)\,$ with $\,D^\sharp f \in \mathscr  L^{\frac{n}{n-1}}\,$  as  Schwartz distributions
      \begin{equation}\label{DistributionalJacobian}
 \Im_f [\phi]  \bydef -\int_\mathbb X \textnormal d f^1 \wedge... \wedge  \textnormal d f^{i-1} \wedge f^i \textnormal d \phi\,\wedge \textnormal d f^{i+1} \wedge .... \wedge \textnormal d f^n\;,\; \textnormal{for}\;\phi \in \mathscr C^\infty_0(\mathbb X)
      \end{equation}
      \end{remark}
      It is clear that $\,\Im_{h_\kappa}  [\phi]  \;\rightarrow    \Im_{h_\infty} [\phi]\, $ for every test function $\, \phi \in \mathscr C^\infty_0(\mathbb X)\,$.
      Then the classical Stokes theorem reveals that  $\, \Im_f [\phi]  = \int_\mathbb X  \phi \;\textnormal{det} Df   \,$,  whenever  $\,f \in \mathscr W^{1,n}(\mathbb X, \mathbb R^n)\,$. But, unfortunately  this is not the case with $\, f = h_\kappa\,$.

\medskip

      One might hope to circumvent this difficulty  by using the  idea of \textit{biting convergence}, first successfully applied to Jacobians by K. Zhang \cite{Zhang}. In fact the above scheme would work if the mappings in question had  non-negative Jacobians. We refer again to Lemma 4.1 in \cite {MullerQiYan}.   Thus, in essence,  the novelty in this example \ref{WeaklyConformal} is that we are unconcerned with the orientation of the mappings. \\

      The following fact, not so easy to prove,  comes to the rescue.
      \begin{proposition} \label{DistributionalConvergence} Under the coercivity condition above, the determinants $\,\textnormal{det} Dh_\kappa\,$  are bounded in the Hardy space  $\,\mathscr H^1(\mathbb X) \subset \mathscr L^1(\mathbb X)\,$. As such,  they converge  to $\,\textnormal{det} Dh_\infty\, \in \mathscr L^1(\mathbb X)\,$ in the sense of distributions,
      \begin{equation}\label{DistributionalLimit} \int_\mathbb X \,\eta(x)\,\textnormal{det} Dh_\kappa(x)\,dx\; \rightarrow  \int_\mathbb X \,\eta(x)\,\textnormal{det} Dh_\infty(x)\,dx\;, \; \textnormal{for}\, \eta \in \mathscr C^\infty_0(\mathbb X)
      \end{equation}
      \end{proposition}
      \end{itemize}
In fact a stronger statement holds; namely, since the Hardy space $\,\mathscr H^1(\mathbb R^n)\,$ is the dual of $\,VMO(\mathbb R^n)\,$ (functions of \textit{vanishing mean oscillation}),   we conclude that formula (\ref{DistributionalLimit}) remains valid  for  the test functions
       $\,\eta \in VMO(\mathbb R^n)\,$ with compact support in $\mathbb X\,$.   For a discussion and results concerning \textit{biting convergence}  see \cite {GrecoIwaniecSubramanian}.
       Let us  demonstrate the lines of reasoning for (\ref{DistributionalLimit}) by using the results of \cite{IwaniecOnninenH1}.

      \begin{proof}  Suppose we are given a Sobolev mapping $\,f : \Omega  \rightarrow \mathbb R^n\,$,  defined on a domain $\,\Omega \subseteq\mathbb R^n\,$,  whose differential (subdeterminants of size $\,1 \times 1\,$)  and its $(n-1) \times (n-1)\,$ subdeterminants are integrable with powers $\,n-1\,$ and $\,\frac{n}{n-1}\,$, respectively,
      \begin{equation}
      \,\Big (\int_\Omega |Df(x)|^{n-1} dx\,\Big) ^{\frac{n}{n-1}}\; +\; \int_\Omega \Big( |D^\sharp f(x)|^{\frac{n}{n-1}} dx \;\Big)\; < \infty
      \end{equation}
      Then
      \begin{equation}
      \,\int_\Omega |\,\textnormal {det} \, Df(x)\,|\, dx\,\;\leqslant     |\!|\, \textnormal {det} \, Df\, |\!|_{\mathscr H^1(\Omega)}\; \leqslant C(n) \int_\Omega |\, D^\sharp f(x)\,| ^{\frac{n}{n-1}}\, dx
      \end{equation}
      Choose and fix a test function $\,\eta \in \mathscr C^\infty_0(\mathbb X)\,$ and  consider the mappings $\, \eta \,h_\kappa : \mathbb R^n \rightarrow \mathbb R^n   \,$ together with  their limit $\, \eta\, h_\infty  : \mathbb R^n \rightarrow \mathbb R^n\,$. It is a routine matter to verify that there is a constant $\,M = M(\eta) < \infty\,$  such that
      \begin{equation}
      \,\Big (\int_{\mathbb R^n} |D(\eta\, h_\kappa)\,|^{n-1}\,\Big) ^{\frac{n}{n-1}}\; +\; \int_{\mathbb R^n}  |D^\sharp (\eta \,h_\kappa) \,|^{\frac{n}{n-1}} \;\;\leqslant   M
      \end{equation}
      for all  $\,\kappa = 1,2, ... $\,,  and also for $\, \kappa =  \infty\,$.
      Now Theorem 1.3  in \cite{IwaniecOnninenH1} can be used as follows.
      There is a subsequence, still  denoted by $\, \eta \,h_\kappa\,$, such that for every $\,\Phi \in\,VMO(\mathbb R^n)\,\,$ it holds that
      \begin{equation}\label{H1weakLimit}\;\;\;\;\;\;\int_{\mathbb R^n} \,\Phi\,\textnormal{det} [ D(\eta h_\kappa) ]\; \rightarrow  \int_{\mathbb R^n} \,\Phi\,\textnormal{det} [ D(\eta h_\infty ) ]\,\;
      \end{equation}
      The meaning of the integrals is understood in the sense of $\,\mathscr H^1 - BMO\,$ duality.

        As a test function in $\,VMO(\mathbb R^n)\,$ we take  $\,\Phi \in \mathscr C^\infty_0 (\mathbb R^n) \,$ that equals $\,\equiv 1\,$ on the support of $\,\eta\,$. The integrals in  (\ref{H1weakLimit})  exist as Lebesgue integrals and are restricted to the domain $\,X\,$, with no factor $\,\Phi\,$.
       $$ \int_{\mathbb X} \,\textnormal{det} [ D(\eta h_\kappa) ]\; \rightarrow  \int_{\mathbb X} \,\,\textnormal{det} [ D(\eta h_\infty ) ]\,\;
        $$
       The remainder of the derivation of  (\ref{DistributionalLimit}) is routine.
      \end{proof}
Finally, having  Proposition \ref{DistributionalConvergence} in hands, we can argue that $\,h_\infty\,$ is in fact an energy-minimal deformation for the functional (\ref{VeryWeaklyConformalFunctional})  in  Example \ref{WeaklyConformal}. A seemingly insignificant fact that the integrand
within the square  brackets in (\ref{VeryWeaklyConformalFunctional})\, is nonnegative for every map in $\,\mathfrak B_\circ\,$, actually  plays a significant role.

\begin{theorem}\label{ExistenceWeaklyConformalMappings}
The energy functional (\ref{VeryWeaklyConformalFunctional}) assumes its infimum in $\,\mathfrak B_\circ$. In effect, the weak $\,\mathscr W^{1, n-1}\,$-limit map $\,h_\kappa \rightharpoonup h_\infty\,$  minimizes the energy.
\end{theorem}

\begin{proof} We are going to make use of the distributional convergence of the Jacobians as expressed by (\ref{DistributionalLimit}).  Choose and fix a test function $\,\eta \in \mathscr C^\infty_0(\mathbb X)\,$ such that $\,0 \leqslant \eta = \eta(x) \leqslant 1\,$. In addition to  (\ref{DistributionalLimit}) we have two inequalities due to the lower semicontinuity of the corresponding polyconvex functionals: \\

 \begin{itemize}
\item [(i)] $\Big(\int_\mathbb X | Dh_\infty |^{n-1}\Big)^{\frac{n}{n-1}}  \;\leqslant\; \liminf_{\kappa \rightarrow \infty}  \Big(\int_\mathbb X | Dh_\kappa |^{n-1}\Big)^{\frac{n}{n-1}}   $\\

\item [(ii)] $ \int_\mathbb X \eta(x)\, | D^\sharp h_\infty (x) \, |^{\frac{n}{n-1}}\; dx \;\;\leqslant \;\; \liminf_{\kappa \rightarrow \infty} \int_\mathbb X \eta(x)\, | D^\sharp h_\kappa (x) \, |^{\frac{n}{n-1}}\; dx  $ \\

\item[(iii)] $ \;\;\;\;\;\int_\mathbb X \,\eta(x)\,\textnormal{det} Dh_\infty (x)\,dx \; =\; \lim_{\kappa \rightarrow \infty} \int_\mathbb X \,\eta(x)\,\textnormal{det} Dh_\kappa(x)\,dx\;$
\end{itemize}

Adding these terms together yields

\begin{eqnarray*}
  \lefteqn{  \mathscr A_\eta [h_\infty] }\\ &  \bydef  & \varepsilon \Big(\int_\mathbb X | Dh_\infty |^{n-1}\Big)^{\frac{n}{n-1}}  +  \int_\mathbb X \eta \,\Big[ \lambda | D^\sharp h_\infty  \, |^{\frac{n}{n-1}}\; - n^{n/n-2} \,\textnormal{det} Dh_\infty \Big] \\&  \leqslant&
   \liminf_{\kappa \rightarrow \infty} \Bigg\{\varepsilon \Big(\int_\mathbb X | Dh_\kappa |^{n-1}\Big)^{\frac{n}{n-1}}  +  \int_\mathbb X \, \eta \,\Big[ \lambda | D^\sharp h_\kappa  \, |^{\frac{n}{n-1}}\; - n^{n/n-2} \,\textnormal{det} Dh_\kappa \, \Big]\;  \Bigg\} \\
   & \leqslant &
      \liminf_{\kappa \rightarrow \infty} \Bigg\{\varepsilon \Big(\int_\mathbb X | Dh_\kappa |^{n-1}\Big)^{\frac{n}{n-1}}  +  \int_\mathbb X \, \,\Big[ \lambda | D^\sharp h_\kappa  \, |^{\frac{n}{n-1}}\; - n^{n/n-2} \,\textnormal{det} Dh_\kappa \, \Big]\;  \Bigg\} \\ &=&
 \liminf_{\kappa \rightarrow \infty} \,\mathscr E^\sharp [h_\kappa] \; = \; \textsf{E}\; \;\;(\textnormal{ a quantity independent of} \;\; 0 \leqslant \eta = \eta(x) \leqslant 1\,) .
\end{eqnarray*}

      In the first line of the above inequalities, the integrand within square  brackets  is nonnegative and $\,\mathscr L^1\,$-integrable. Therefore, by letting $\,\eta\,$ approach $\,\equiv 1\,$ point-wise,  we may (and do) pass to the supremum. This gives us the desired estimate
      $$
     \mathscr E^\sharp[h_\infty] \,=\, \mathscr A_{_{\equiv 1}}[h_\infty] \; \leqslant  \textsf{E} \,\;\;  (\textnormal{the infimum energy})
      $$
\end{proof}
\begin{remark} It seems likely that Theorem \ref{ExistenceWeaklyConformalMappings} remains valid for $\,\lambda = 1\,$ and $\,\varepsilon > 0\,$ as well. To gain this, however,   one must devise a  \textit{mean coerecivity} (as opposed to point-wise coerecivity) for  the energy integrals  of $\,|\,D^\sharp h\,|^{\frac{n}{n-1}}\,$; say,   by analogy with  (\ref{MeanCoercivity}). 

\medskip

The very borderline case of $\,\varepsilon = 0\,$ and $\,\lambda = 1\,$  (so very weak conformality) seems to be of interest for further studies.
\end{remark}

\section{Free Lagrangians and frictionless deformations.}

It is of great interest to study the relationship between topology and analysis in the setting of extremal problems.  In particular one is often interested in obtaining a ``nice'' mapping in the homotopy class of a given mapping between spaces.  After the Riemann mapping theorem perhaps the best known example of this are questions around the existence of harmonic mappings in a homotopy class of mappings.  This problem was pioneered in the work of Eells and Sampson from 1964 \cite{ES}, covered in the later long papers of Eels and Lemaire \cite{EL,EL2} and has since found wide application in very many areas of mathematics,  far too numerous to note.  To formalise this problem in the setting of this article we come to the notion of a free Lagrangian. 

\medskip

 Let us say to begin with that  \textit{free Lagrangians} are quite like null Lagrangians, and we represent them by the symbol
\[ \pmb{\mathscr F}_{\! L}(\mathbb X, \mathbb Y)\,\] 
 It has to be emphasised that this notation is refers to a given pair of domains $\mathbb X,  \mathbb Y \subset \mathbb R^n$ of the {\em same} topological type.  That is they  are homeomorphic. As opposed to the situation considered for null Lagrangians, the concept of free Lagrangians concerns nonlinear differential $\,n\,$-forms $\,\mathbf{L}(x, h, Dh) \,dx\,$, defined for Sobolev homeomorphisms $\, h : \mathbb X \onto \mathbb Y\,$,  whose integral means (which we usually call the energy of $\,h\,$) depend only on the homotopy class of $\,h\,$,  and not on its boundary values.  This of course leads to frictionless type problems.
 \begin{definition}
 Suppose that a given energy integral
 \begin{equation}
 \mathscr F[h] \bydef  \int_\mathbb X \,\mathbf{\mathbf{F}}(x, h, Dh) \,dx\, \,,\;\; \textnormal{where} \;\;\;\;\;  \mathbf F : \mathbb X \times \mathbb Y \times \mathbb R^{n\times n}  \rightarrow \mathbb R
 \end{equation}
 converges for all Sobolev mappings of class \,$\mathscr W^{1,p}(\mathbb X, \mathbb Y)\,$.
We say that the differential $\,n\,$-form $\,\mathbf{\mathbf{F}}(\cdot, \cdot, \cdot) \,dx \,$ is a Free Lagrangian if
$$
\mathscr F[h_1]  = \mathscr F[h_2] \;\;,\;\textnormal{whenever}\;\; h_1\simeq h_2 \,  $$
That is, whenever the Sobolev homeomorphisms  $\, h_1, h_2 : \mathbb X  \onto \mathbb Y\,,\, \textnormal{in}  \;\, \mathscr W^{1,p}(\mathbb X, \mathbb R^n ) \,$, are homotopy equivalent.
 \end{definition}
 
 It is important to note here that we are only considering homotopy equivalence {\em between surjections} (sometimes perhaps with other restrictions such as homeomorphisms).  If $\X,\Y$ are balls,  then any continuous map $\X\to\Y$ is certainly homotopic to a constant map.

 We shall now make this concept clear with selected examples and illustrate how this leads to questions of the existence of frictionless energy minimal deformations.

\subsection{Examples}
\begin{itemize}
\item [$\mathscr F_{\!L}^1$ )]   (\textbf{A volume form in $\,\mathbb X\,$}) 
This simple example  is still useful (as we will see)  and is given by an integrand independent of $\,h\,$ :
\begin{equation}\label{VolumeFormInX}
\mathbf{\mathbf{F}}(x,   y,   \xi ) \,dx \; = \mathbf F(x) \,dx\;\;,\;\;\textnormal{where} \;\,\mathbf F \in \mathscr L^1(\mathbb X)
\end{equation}
\item [$\mathscr F_{\!L}^2$ )] (\textbf{A pullback of a volume form in $\,\mathbb Y\,$})\\
An  example, to some extent dual  to the above, is the pullback of a volume form  $\,\Phi(y) \, d y \,$ with $\,\Phi \in \mathscr L^1(\mathbb Y)\,$,
\begin{equation}\label{VolumeFormInY}
\mathbf{\mathbf{F}}(x,   y,   \xi ) \,dx \; = \Phi(y) \,\textnormal{det}\, \xi\,\,dx\;\;
\end{equation}
\end{itemize} 
We note the identity
\begin{equation}
 \mathscr F[h] \;=\; \int_\mathbb X \,\Phi(h)\,J(x,h)\, dx = \int_\mathbb Y \Phi(y)\,\textnormal d y\;
\end{equation}
which holds for all orientation preserving homeomorphisms $\,h : \mathbb X \onto \mathbb Y\,$ in the
Sobolev space ${\mathscr W}^{1,n}(\mathbb X, \mathbb Y)\,$. In other words, using the notation of exterior algebra,  the $\,n\,$-form  $\,\Phi(h)\, dh^1 \wedge ... \wedge dh^n
\,$ is  a free Lagrangian. Applications are still limited because the associated energies  do not recognise the geometric shape of the domains $\,\mathbb X\,$ and $\,\mathbb Y\,$ only their volumes.
A myriad of differential expressions, representing free Lagrangians,  are yet to be found. The next two examples (again dual to each other)
provide very useful free Lagrangians  for a pair of round annuli 
\[ \mathbb X = {\mathbb A}=\{ x\, : \; \; \; r <|x| < R\}, \quad {\rm and, } \quad
 \mathbb Y ={\mathbb A}^\ast =\{ y\, :\; \; \; r_\ast < |y| < R_\ast \}\]
\begin{itemize} 
\item [$\mathscr F_{\!L}^3$ )] (\textbf{Radial increment of $\,|h| \,$} $ \,;\,  \pmb{\mathscr F}_{\! L}(\mathbb A, \mathbb A^\ast)\, $\; )
\end{itemize}

\begin{proposition}  \label{PR1}
The following differential $n$-form $\,\mathbf F(x, h, Dh) \,\textnormal dx
,$ :
\begin{equation}
 \frac{\left(d|h|\right)\wedge \star d|x|}{|h|\, |x|^{n-1}} \,\bydef\, \sum_{i=1}^n \frac{x_i\, dx_1 \wedge ... \wedge dx_{i-1}\wedge d|h|\wedge dx_{i+1}\wedge ... \wedge d x_n}{|h|\, |x|^n}
\end{equation}
is  a free Lagrangian in the homotopy  class of orientation preserving homeomorphisms that preserve the order of the
boundary components of the annuli ${\mathbb A}$ and ${\mathbb
A}^\ast$. Precisely, for such homeomorphisms of Sobolev class  $\,\mathscr
W^{1,1}(\mathbb A\,, \,\mathbb A^\ast)\,$, we have
\begin{equation}\label{Eq178}
\int_{\mathbb A} \mathbf F(x, h, Dh) \,\textnormal dx \,=\,  \int_{\mathbb A} \frac{d|h|\; \wedge \star d|x|}{|h|\,
|x|^{n-1}}=\textnormal{Mod}\, {\mathbb A}^\ast\; \bydef \log \frac{R^*}{r^*}\;\;
\end{equation}
\end{proposition}
Here we have used the Hodge star duality  operator in the exterior algebra; namely,   $\,\ast: \Lambda ^1(\mathbb R^n)\, \onto \Lambda^{n-1}(\mathbb R^n)\,$.

Observe that the function $\,|h|: {\mathbb A}
\rightarrow (r_\ast, R_\ast)\,$ extends continuously to the closure of
${\mathbb A}$. That $\,h\,$ preserves the order of the spherical boundary components simply means that $|h(x)|=r_\ast$ for $|x|=r$ and
$|h(x)|=R_\ast$ for $|x|=R\,$.  Accordingly,

$$\,\mathbf{\mathbf{F}}  \,dx \,=\; \frac{d|h|\; \wedge \star d|x|}{|h|\,
|x|^{n-1}}\;  \in   \pmb{\mathscr F}_{\! L}(\mathbb A, \mathbb A^*)\, $$
\begin{itemize}
\item [$\mathscr F_{\!L}^4$ )] (\textbf{Spherical derivatives of } $\,h \,$  $ \,;\,   \pmb{\mathscr F}_{\! L}(\mathbb A, \mathbb A^\ast)\, $\;)
\end{itemize} 
Another free Lagrangian in $\,\mathbf{\mathbf{F}}(\cdot, \cdot, \cdot) \in   \pmb{\mathscr F}_{\! L}(\mathbb A, \mathbb A^*)\, $, dual to that in Proposition \ref{PR1}, exploits topological degree of the mappings $\,h : \mathbb S^{n-1}_t  \rightarrow \mathbb R^n\setminus\{0\} \,$ restricted to the concentric spheres of radii $\,t \in (r , R) \,$. The degree is equal to  1 on every sphere, which yields

\begin{proposition} \label{Le333}
The following  differential $n$-form  $\,\mathbf F(x, h, Dh) \,\textnormal dx \;: $
\begin{equation}
 \frac{d|x|}{|x|}\wedge h^\sharp \omega = \sum_{i=1}^n \frac{h^i\, dh^1 \wedge ... \wedge dh^{i-1}\wedge d|x|\wedge dh^{i+1}\wedge ... \wedge d h^n}{|x|\, |h|^n}
\end{equation}
is a free Lagrangian in the class of all orientation preserving homeomorphisms $h \in {\mathscr W}^{1,n-1}({\mathbb A}, {\mathbb
A}^\ast)$. Precisely, we have
\begin{equation}\label{331}
\int_{\mathbb A} \mathbf F(x, h, Dh) \,\textnormal dx \,=\,\int_{\mathbb A} \frac{d|x|}{|x|}\wedge h^\sharp \omega =  \textnormal{Mod}\, {\mathbb A} \; \bydef \log \frac{R}{r}
\end{equation}
\end{proposition}
Here $\,h^\sharp \omega\,$ stands for the pullback of the $\,(n-1)\,$-area form  defined in  $\,\mathbb A^\ast\,$; namely,
$$\,\omega (y) = \sum_{i=1}^n (-1)^i\,\frac{y^i\, dy^1 \wedge ... \wedge dy^{i-1}\wedge  dy^{i+1}\wedge ... \wedge d y^n}{\, |y|^n}\;$$
This is none other than the $\,(n-1)\,$ area form on any $\,(n-1)\,$-closed surface in $\,\mathbb A^\ast\,$ that   is  homologous to $\,\mathbb S^{n-1}\,$.  A far more detailed exposition, suited to this example, is presented   in    \cite[Chapters 6 and 7]{IOan}.
 
\subsection{The Nitsche frictionless problem}\label{NTFP} Given a pair planar annuli $\,\A= \{x \in \C \colon r< \abs{x} <R\}\,$ and $\,\A^\ast= \{y \in \C \colon r_\ast< \abs{y} <R_\ast\}\,$,
the objective is to minimize the Dirichlet energy subject to  Sobolev homeomorphisms $h \colon \A \onto \A^\ast$ in  $\W^{1,2} (\A , \C)$.
\[\mathscr E_2 [h] = \int_\A \abs{Dh(x)}^2 \, \dtext x  \;  \]
 Note that no boundary values of these homeomorphisms are prescribed, whence the name frictionless is given to this problem. We denote this class of mappings by $\mathscr H^{1,2} (\A, \A^\ast)$.

 Naturally, polar coordinates
\begin{equation}
x= t\, e^{i\theta} \, , \hskip0.5cm r< t<R \; \; \textnormal{ and }\; \; 0 \leqslant \theta < 2 \pi
\end{equation}
are best suited.  The radial (normal) and angular (tangential) derivatives of $h$ are defined by
\begin{equation}
h_{_N} (x) = \frac{\partial h(te^{i \theta})}{\partial t} \, , \hskip1cm t=|x|
\end{equation}
and
\begin{equation}
h_{_T}  (x) =  \frac{1}{t}\frac{\partial h(te^{i \theta})}{\partial \theta} \, , \hskip1cm t=|x|
\end{equation}
The stored energy integrand takes the form
\[\abs{Dh(x)}^2 = \abs{h_{_N} (x)}^2 + \abs{h_{_T} (x)}^2 \,   \]
This also provides an effective formula for the Jacobian determinant
\[J_h(x) = \det Dh(x) =\im (\overline{h_{_N} } {h_{_T}} ) \le \abs{h_{_N} } \, \abs{h_{_T} }\, .  \]

Our free Lagrangians $\mathscr F_{\!L}^1$ )\,, \,..\,,\,$\mathscr F_{\!L}^4$ ) in $\, \pmb{\mathscr F}_{\! L}(\mathbb A, \mathbb A^\ast)\, $ can easily be stated using polar coordinates in even slightly greater generality.   The main players and their energy integrals for $h\in \mathscr H^{1,2} (\A, \A^\ast)$  are:

\begin{eqnarray*}
 \mathscr F_{\!1})  && \,\int_\mathbb A M(x)\, dx\,  , \quad \quad M \in \mathscr L^1 (\A)  \\
 \mathscr F_{\!2}) && \,\int_\mathbb A N(\abs {h}) J_h(x)\,\textnormal dx\,=\, 2\,\pi \int_{r_\ast}^{R_\ast} N(s) \,s\,\textnormal ds
\,\\
 \mathscr F_{\!3}) && \,\int_\mathbb A  A(\abs{h}) \frac{\abs{h}_N}{\abs{x}}\,\textnormal dx\,\,=\, 2\pi \int_{r_\ast}^{R_\ast} A(s)\,\textnormal ds  \quad \quad A\in \mathscr L^1 (r_\ast, R_\ast)\\
\mathscr F_{\!4}) && \,\int_\mathbb A  B \big( |x|\big) \im \frac{h_T}{h}\,\textnormal dx\,\, =\; 2\pi \int_{r}^{R} B(t)\, dt,\quad \quad  B \in {\mathscr L}^1 (r , R).
\end{eqnarray*}

\subsection{Energy minimizers among radial mappings} It is natural to first look at the radial mappings as candidates for energy-minimizers. However, this expectation is far from being  guaranteed. In spite of the radial symmetry of the annuli and the invariance of the Dirichlet energy under rotations of $\,\mathbb A\,$ and $\,\mathbb A^\ast\,$, such a lack of symmetry of the energy-minimizers has,  quite surprisingly,  been confirmed already  in the analogous Nitsche problem in  dimensions $\,n \geqslant 3\,$. Nevertheless, the extremals within the radial mappings give us  the  pinpoint of  \textit{free-Lagrangian} to solve the minimization problem in full generality. The radial mapping
\begin{equation}\label{GeneralRadialStretching}
h_\circ (x) = H(\abs{x}) \frac{x}{\abs{x}}, \qquad \textnormal{where } H \colon [r,R] \onto [r_\ast, R_\ast]  \,
\end{equation}

If one seeks harmonic radial mappings then 
\[  \Delta h_\circ (x) = h_{tt} + \frac{1}{t} h_t \,+\,t ^{-2}  h_{\theta \theta }  = 0,\]
where  $ \, x= t\,e^{i \theta} \, $.  Then  $\,H = H(t)\,$ must satisfy the Euler's ordinary differential equation
\[ t^2\,\ddot{H}(t)\, +\, t\,  \dot{H}(t)\; -\,  \,H(t)  \, = 0\,\;\textnormal{for}\, r< t < R \]
Its two fundamental  solutions $\, t\,$ and $\, \frac{1}{t} \,$ generate all solutions  
\[ \,H(t)  \,=\, a \,t \, + \, b/t .\]
To ensure the boundary constraints, $\, H(r) = r_\ast\,$  and $\, H(R) = R_\ast \,$,  we must set,
\begin{equation}\label{ExplicitFormulas}
H(t)  \,=\, a\, t \, + \, b/t \;,\;\textnormal{where}  \;\; a = \frac{R R_\ast  - r r_\ast}{R^2 - r^2}\;\;\textnormal{and}\;\; b = \frac{R^2 r r_\ast \,-\, r^2 R R_\ast}{R^2 - r^2}
\end{equation}
It is advantageous to transfer  the above  equation to the first order ODE by simply   multiply by the integrating factor $\,-2 \dot{H}(t)\,$.  We   obtain the \textit{characteristic equation}
\begin{equation}\label{eq:Nitschesolution}
\mathcal L [H] \bydef  H^2 -t^2 \dot{H}^2 \equiv c \,,  \mbox{where $c\in \mathbb R$ is a constant.}
\end{equation}
Note that, as opposed to the second order Laplace equation,  the first order characteristic equation admits an additional solution,  namely, a constant function.

\medskip

We shall be concerned with monotone $\mathscr C^1$-solutions $H\colon [r,R] \onto [r_\ast, R_\ast]$,  thus having  $\,\dot{H}(t) \ge 0\,$. This includes  solutions that are partially constant.  In particular, they may squeeze but not fold subintervals. Under these assumptions the respective radial mappings $\,h_\circ\,$  become uniform limits of homeomorphisms, as desired in the weak formulation  of the principle of noninterpenetration of matter discussed earlier.

\begin{remark} Equation (\ref{eq:Nitschesolution}), for such solutions,  is none other than the variational equation of the minimization problem when confined to the radial mappings. This fact, though natural to expect,  is not automatic and we will exploit it.\\
We will not use the explicit formulas (\ref{ExplicitFormulas}), but only the characteristic equation (\ref{eq:Nitschesolution}). This means to our arguments have can be developed for frictionless problems in which one can predict the PDEs for the energy-minimisers,  if not their explicit solutions.    \\
\end{remark}

There is a useful quantity associated with the radial mappings called the {\it elasticity of stretching}
\[\eta_{_H} (t)\, \bydef\, \frac{t \dot {H}(t)}{H(t)}  \]

All $\mathscr C^1$-solutions fall into three categories. If $c$ is the constant at (\ref{eq:Nitschesolution}) we say that:
\begin{itemize}
\item $H$ is conformal if  $c = 0$, equivalently ${\eta_{_H} (t)} =1$ iff  $\frac{R_\ast}{r_\ast} =\frac{R}{r}$
\item $H$ is expanding if  $c < 0$, equivalently ${\eta_{_H} (t)} >1$ iff  $\frac{R_\ast}{r_\ast} > \frac{R}{r}$
\item $H$ is contracting if $c > 0$, equivalently ${\eta_{_H} (t)} < 1$  iff  $\frac{R_\ast}{r_\ast} < \frac{R}{r}$

\end{itemize}
The characteristic equation~\eqref{eq:Nitschesolution} has an injective solution if and only if
\begin{equation}\label{Nitschebound}
\frac{1}{2} \left( \frac{R}{r} + \frac{r}{R}\right) \le \frac{R_\ast}{r_\ast}
\end{equation}
We call the set of values for whiich (\ref{Nitschebound}) holds the  Nitsche range,  it includes the expanding case.
 We reserve the notation $F(\tau) \bydef H^{-1}(\tau) $ for  $r_\ast < \tau < R_\ast\,$.  Thus the characteristic equation reads as
\begin{equation}\label{eq:NitscheInverseSolution}
\left(\frac{F}{\dot{F}}\right)^2 - \tau^2 \equiv c\, \quad\quad F(\tau) =  \frac{\tau  +  \sqrt{\tau^2 - c} }{ 2 \, a}  .
\end{equation}
When the reverse inequality to (\ref{Nitschebound}) holds, we say the data lies beyond the Nitsche range.  We consider a $\,\mathscr C^{1,1}\,$-solution $H\colon [r,R] \onto [r_\ast, R_\ast]\,$  of ~\eqref{eq:Nitschesolution} defined by the rule.
\begin{equation}\label{eq:Nittschemap}
H(t) = \begin{cases} r_\ast  &  t\in [r, \rho] \\
H_\rho (t)  &  t\in [\rho, R] \end{cases} \;, \;\;\;\;  \textnormal{ $ \,\rho\,$ is determined by  } \quad  \frac{1}{2} \left( \frac{R}{\rho} + \frac{\rho}{R}\right) = \frac{R_\ast}{r_\ast}
\end{equation}
Here \,,\; $\,H_\rho(t) = \frac{r_\ast}{2} \big( t + \frac{\rho}{t}\big)     \,$,\;    is an increasing solution $H_\rho \colon [\rho, R] \onto [r_\ast, R_\ast]$,  of the characteristic equation ~\eqref{eq:Nitschesolution}.

\subsection {The  conformal case} In this case the pullback of the area form alone is sufficient to identify the energy-minimal mappings. Precisely, using  $\mathscr F_{\!2}$) ,  we obtain.

\[\int_\mathbb A \abs{Dh(x)}^2 \textnormal{d} x  = \,\int_\mathbb A \left|h_N\right|^2 +\left|h_T\right|^2  \ge 2\, \int_\mathbb A  \abs{h_N} \abs{h_T} \ge 2 \, \int_\mathbb A J_h(x)\, dx  =  2 |\mathbb A^\ast | \]
Equality occurs only for a similarity transformation of $\,\mathbb A\,$ onto $\,\mathbb A^\ast\,$ (scalar multiple of a rotation) .

\subsection{The expanding case}
Choose and fix a radial mapping $h_\circ = H\big(|x|\big) \frac{x}{|x|}$, where $H$ solves Equation~\eqref{eq:Nitschesolution} with $\,c <0\,$.  Explicitly, in complex notation, we have the formula
\begin{equation}\label{ComplexNitcheMap}
h_\circ(z) = a z + \frac{b}{\bar z}.
\end{equation}
 It is true that $\,h_\circ\,$ turns out to be  the energy-minimal solution among radial mappings, but we shall not exploit this property; equation~\eqref{eq:Nitschesolution} is sufficient.  Now suppose we are given an arbitrary mapping $h\in \mathscr H^{1,2} (\A, \A^\ast)$. We shall derive a series of sharp estimates involving   $h$ and $Dh$, each of which becomes an equality if $h=h_\circ$ - the radial Nitche map.  Let us introduce the following functions; first defined for  $\,r_\ast \leqslant \tau \leqslant R_\ast\,$ by the rule
 \[ p(\tau) = \eta_F (\tau)=  \frac{\tau \, \dot{F}(\tau)}{F(\tau)} <1, \quad\quad \Big(\;p(\tau) = \frac{\tau \sqrt{\tau^2 - c} \; +\; \tau^2 }{\tau \sqrt{\tau^2 - c} \; +\; \tau^2 - c } \Big) \]
We note that $\,p(|h(x)|) = \frac{|h_N(x)|}{|h_T(x)|}\,$ for $\,h \bydef h_\circ\,$. The second function of two variables is defined by
 \[ A(t,\tau) \bydef \frac{F(\tau)}{t \, \dot{F}(\tau)}\, , \hskip1cm  \textnormal{for}\; r \leqslant t \leqslant R \; \;\textnormal{and}  \; \; r_\ast \leqslant \tau \leqslant R_\ast \]
We note that the function $\, x  \mapsto A(|x|,|h(x)|)$ is equal to $\left|h \right|_N(x) \,$ in case that $\,h=h_\circ (x)\,$.   

To proceed we make identify algebraic inequalities which lead  to the lower bounds of the integrand by means of free-Lagrangians. In our case,  the  following point-wise inequality holds whenever $\,0 \leqslant p \leqslant 1\,$ and $\,A \geqslant 0\,$.
\begin{eqnarray*}
|Dh|^2 & = &  \left|h_N\right|^2 +\left|h_T\right|^2 \geqslant (1-p^2) \left|h_N\right|^2\;  +\;  2p \left|h_N\right|\, \left|h_T\right|  \\
&\geqslant &  (1-p^2)\,  2\, A \, \left|h\right|_N\,  - \, (1-p^2)\, A^2 \; + \; 2\, p\, J_h
\end{eqnarray*}
Now, according to~\eqref{eq:NitscheInverseSolution} we have $(1-p^2)\, A^2=  c |x|^{-2} \,$ and so
\[|Dh|^2 \ge 2 \left(  \frac{F(\abs{h})}{\dot{F}(\abs{h})}   -  \frac{\abs{h}^2 \dot{F}(\abs{h})}{{F}(\abs{h})}  \right) \,   \frac{\abs{h}_N}{\abs{x}}  -\; \frac{c}{\abs{x}^2}\; +\; 2\, p(\abs{h})\, J_h \]
Here the right hand side consists of free-Lagrangians and equality occurs for the radial Nitsche map $h_\circ\,$. Hence
\begin{equation}
\int_\mathbb A  |Dh(x)|^2\, dx \;\;\geqslant \;\; \int_\mathbb A  |Dh_\circ(x)|^2\, dx \;\;,\;\;\textnormal{as desired.}
\end{equation}

\subsection {The contracting case} In contrast to the expanding case, in the contracting case apply spherical free-Lagrangians  $\,\mathscr F_{\!4}) \,$ in place of the radial free-Lagrangian $\,\mathscr F_{\!3}) \,$. We discuss two subcases.

\subsubsection{Annuli still within the Nitsche range at (\ref{Nitschebound})}   In this subcase the radial solution $\,h_\circ\,$ still remains injective. The first step is to choose a good point-wise inequality. Precisely, for all  parameters $\,0\leqslant q \leqslant 1\,$ and $\,B \geqslant 0\,$ it holds, as is easily verified, that
\begin{eqnarray*}
|Dh|^2 &=& \left|h_N\right|^2 + \left|h_T\right|^2 \geqslant (1-q^2) \left|h_T\right|^2 + 2 q \left|h_N\right| \, \left|h_T\right| \\
& \geqslant & (1-q^2)\, 2B \left|h_T\right| \, -\,  (1-q^2)\, B^2 \, +\, 2\, q \, J_h \\
& \ge & (1-q^2)\, 2B \abs{h} \im  \frac{h_T}{h} \, -\,  (1-q^2)\, B^2 \, +\, 2\, q \, J_h
\end{eqnarray*}
 Then, in the second step,  we   take for $\,q\,$ and $\,B\,$ the following functions
\begin{eqnarray}
q&=&q(\abs{h}) = \frac{1}{\eta_F (\abs{h})}= \frac{F(\abs{h})}{\abs{h} \, \dot{F}(\abs{h})} <1,\\
B&=&B(\abs{x}, \abs{h}) =    \left( \frac{H^2(\abs{x}) }{\abs{x}} - \abs{x} \, \dot{H}^2 (\abs{x})  \right)   \frac{1}{\abs{h} \, (1-q^2(\abs{h})  )} \, .
\end{eqnarray}
Now since  $\abs{h}^2 \, (1-q^2(\abs{h}) \equiv c $ by~\eqref{eq:NitscheInverseSolution} we obtain the desired lower bound on $\,|Dh|^2\,$ by free-Lagrangians.
\begin{equation}\label{eq:contracting}
|Dh|^2 \ge 2  \left( \frac{H^2(\abs{x}) }{\abs{x}} - \abs{x} \, \dot{H}^2 (\abs{x})  \right)  \im  \frac{h_T}{h}  -  \frac{1}{c}  \left( \frac{H^2(\abs{x}) }{\abs{x}} - \abs{x}   \dot{H}^2 (\abs{x})  \right)^2
 +  \,  \frac{2\,J_h}{\eta_{_F} (\abs{h})}
\end{equation}

Here again the right hand side consists of free-Lagrangians and equality occurs for the radial Nitsche map $h_\circ\,$. Hence
\begin{equation}
\int_\mathbb A  |Dh(x)|^2\, dx \;\;\geqslant \;\; \int_\mathbb A  |Dh_\circ(x)|^2\, dx \;
\end{equation}

\subsubsection{Annuli beyond  the Nitsche bound} This means that
\[\frac{1}{2} \left( \frac{R}{r} + \frac{r}{R}\right) > \frac{R_\ast}{r_\ast} \]
 It is exactly in this subcase that the radial solution $\,h_\circ\,$ fails to be  injective.
A plausible candidate for the energy-minimal deformation is the  squeezing/stretching radial mapping $\,h_\circ (x) = H(\abs{x}) \frac{x}{\abs{x}}\,$, where $\,H \colon [r,R] \onto [r_\ast, R_\ast]\,$ is given by~\eqref{eq:Nittschemap}. \\
For $\rho < \abs{x} <R$ we apply~\eqref{eq:contracting} and obtain a lower bound in terms of free-Lagrangians
\[
|Dh|^2 \ge 2 \, \left( \frac{H_\rho^2(\abs{x}) }{\abs{x}} - \abs{x} \, \dot{H}_\rho^2 (\abs{x})  \right)  \im  \frac{h_T}{h} -\frac{1}{c}  \left( \frac{H_\rho^2(\abs{x}) }{\abs{x}} - \abs{x} \, \dot{H}_\rho^2 (\abs{x})  \right)^2
 + 2 \,  \frac{J_h}{\eta_{_{F_\rho}} (\abs{h})}\]
The remaining free-Lagrangians lower bound for $\,r< \abs{x} \le \rho\,$ is just as easy to verify.
\[ |Dh|^2= \left|h_N\right|^2 + \left|h_T\right|^2  \ge 2B \,r_\ast \im \frac{h_T}{h}- B^2\;\;\textnormal{with}\;\; B(|x|) = r_\ast |x|^{-1} \]
The point is that both lower bounds  also become equalities for $\,h_\circ\,$. In conclusion,
\begin{equation}
\int_\mathbb A  |Dh(x)|^2\, dx \;\;\geqslant \;\; \int_\mathbb A  |Dh_\circ(x)|^2\, dx \;,\;\;\;\textnormal{as well}
\end{equation}
All the above cases summarize as follows
\begin{theorem}
The frictionless deformations $\,h :\mathbb A \onto \mathbb A^\ast \,$ between annuli assume their minimum Dirichlet energy among radial mappings.
\end{theorem}
Actually, our proofs easily give an even more precise statement; namely, the energy-minimal mappings are unique up to a rotation of the variable $\,x \in \mathbb A\,$ or, equivalently, a rotation of $\,y\in \mathbb A^\ast\,$.

 \section{Mappings of $\,\mathscr L^p\,$-minimal distortion}
  
In relation to the previous problem, we have already noted in the introduction at (\ref{Kdef}) and the preamble, that should $h$ be a homeomorphism of Sobolev class,  then
\begin{equation}
\int_\A |Dh|^2 \,dz = \int_{\A^*} \IK(z,f) \,dz, \quad \quad f = h^{-1}:\A^*\to \A.
\end{equation}

If $\Phi:V\onto W$ is a conformal mapping between planar domains,  and $f:U\onto V$ is a homeomorphism of Sobolev class $\W^{1,p} (U,V)$, $p\geq 1$, then it is an elementary calculation that
\begin{equation}\label{kid} \IK(z,\Phi\circ f) = \IK(z, f)\end{equation} and hence
\begin{equation}
 \int_{U} \IK(z,\Phi\circ f) \,dz = \int_{U} \IK(z, f)\,dz .
\end{equation}
Let $\Omega$ be a doubly connected planar domain.  The Riemann mapping theorem for doubly connected domains tells us that there  is a single conformal invariant,  the modulus of $\Omega$ -- denoted $mod(\Omega)$,  which identifies a conformally equivalent round annulus $\A=\A(r,R)$ and $\log R/r = mod(\A)=mod(\Omega)$. We can now restate the results of \S \ref{NTFP} as follows. 

\begin{theorem}[{\bf p=1}]\label{p=1} Let $\Omega$ be a doubly connected planar domain with $s=mod(\Omega)$ and $\A=\A(1,R)$, $r=mod(\A)$. Then there is a homeomorphism $f:\A \onto \Omega$ of Sobolev class $\W^{1,1} (\A,\Omega)$ minimising  $\int_{\A} \IK(z, f)\,dz$
if and only if $\cosh(s) \leq  e^r$.  
This minimiser is a diffeomorphism and is unique up to conformal automorphisms of $\Omega$.
\end{theorem}
This is the Nitsche bound (\ref{Nitschebound}). There some subtlety here regarding the regularity of $h=f^{-1}$.  In particular is $h\in \W^{1,2} (\Omega,\A)$ ?  This is discussed in \cite[\S 21]{AIMb} and \cite{HK}.  Despite being of considerable interest,  we set aside discussion of these issues in this article.
In contrast to Theorem \ref{p=1} we have the following theorem concerning extremal quasiconformal mappings.

\begin{theorem}[{\bf p=$\infty$}]\label{p=infty} Let $\Omega$ be a doubly connected planar domain with $0< mod(\Omega)<\infty$. Then for every $0<mod(\A)<\infty$,  there is a homeomorphism $f:\A \onto \Omega$ of Sobolev class $\W^{1,2} (\A,\Omega)$ minimising 
$ \|\,\IK(z, f)\,\|_{L^{\infty}(\A)} $
This minimiser is a diffeomorphism and is unique up to conformal automorphisms of $\Omega$.
\end{theorem}
One can reduce this to the case $\Omega=\A(1,e^s)$ using (\ref{kid}) and $\A=\A(1,e^r)$.  Then identify the extremal quasiconformal mapping as the radial map $z\mapsto z|z|^{\alpha-1}$, $\alpha=\log(s)/\log(r)$.

\medskip

There is a surprising difference here between Theorem \ref{p=1} and Theorem \ref{p=infty}.  Namely there is always and extremal quasiconformal mapping, but a minimiser of mean distortion exists only within a range of moduli -- precisely the Nitsche range.  This leads naturally to the question of what happens for $1<p<\infty$ when we minimise 
\[ \int_{\A} \IK^p(z, f)\,dz \]
The following answer is a little surprising,  \cite{MM}, (this also was recently developed in the context of free-Lagrangians in \cite{IOp} by the first and third authors,  but here we give a different,  but actually equivalent, approach).  The result suggests an interesting critical phase type phenomena.  When $p<1$,  apart from the identity map,  minimizers never exist.  When $p=1$ we observe Nitsche type phenomena; minimisers exist within a range of conformal moduli determined by properties of the weight function and not otherwise.  When $p>1$ minimisers always exist. We will now go through the proof of this.  In \cite{MM} rather more is proved,  namely there the weighted $L^p$-mean distortion is considered.   As noted above the problem is reduced to considering deformations $f:{\mathbb A}_1 \to {\mathbb A}_2$ between round annuli.
 
\subsection{Mappings of finite distortion}

A homeomorphism $f:\Omega\to \Omega'$ between planar domains of Sobolev class $\mathscr W^{1,1}_{loc}(\Omega,\Omega')$ has finite distortion if the Jacobian determinant $J(z,f)$ is nonnegative and there is a function $\IK(z,f)$ finite almost everywhere such that
\[ |Df(z)|^2 \leqslant \IK(z,f) \; J(z,f). \]
Recall function $\IK(z,f)$ is the   {\em distortion} of the mapping $f$.  We saw at (\ref{Kdef}) that $\IK$ is a measure of the anisotropic local stretching.

Mappings of finite distortion  are generalisations of quasiconformal homeomorphisms and have found considerable recent application in geometric function theory and nonlinear PDEs, \cite{AIMb}.   We define annuli
 \[ {\mathbb A}_1 = \{1\leqslant |z| \leqslant R \}, \hskip30pt  {\mathbb A}_2 = \{1\leqslant |z| \leqslant S \} \]
with moduli $\sigma_1=\log(R)$ and $\sigma_2=\log(S)$.  We consider homeomorphisms of finite distortion $f:{\mathbb A}_1\to{\mathbb A}_2$ mapping the boundary components to each other,   
 \[ f(\{|z|=1\})=\{|z|=1\}, \hskip20pt {\rm and} \hskip20pt f(\{|z|=R\})=\{|z|=S\}. \]
On the annulus ${\mathbb A}_1$ place a positive weight
$\eta:{\mathbb A}_1 \to {\mathbb R}_+ $.  In polar coordinates 
\begin{equation}
 f_z = \frac{1}{2} \; e^{-i\theta} \left( f_\rho - \frac{i}{\rho} \; f_\theta \right), \;\;\; \;\;
 f_{\bar z} = \frac{1}{2} \; e^{i\theta} \left( f_\rho + \frac{i}{\rho} \; f_\theta \right) 
\end{equation}
and $|f_z|^2+|f_{\bar z}|^2 = \frac{1}{2}(|f_\rho|^2+\rho^{-2} |f_\theta|^2)$, 
$J(z,f) = |f_z|^2-|f_{\bar z}|^2 = \frac{1}{\rho}\; {\rm Im}(f_\theta \overline{f_\rho})$
which together yield
\begin{equation}\label{2.2}
\IK(z,f) = \frac{|f_z|^2+|f_{\bar z}|^2}{|f_z|^2-|f_{\bar z}|^2} = \frac{\rho |f_\rho|^2+\rho^{-1} |f_\theta|^2}{2 {\rm Im}(f_\theta \overline{f_\rho})}.
\end{equation}
Given a convex function $\varphi:[1,\infty) \to [0,\infty)$ a Nitsche type problem asks us to establish the existence or otherwise of a minimizer (or perhaps stationary point) of the functional 
\begin{equation}\label{nitprob}
f \mapsto \int_{{\mathbb A}_1} \varphi(\IK(z,f)) \; \eta(z) \; |dz|^2.
\end{equation}
Thus we seek a deformation of the annulus ${\mathbb A}_1$ to ${\mathbb A}_2$ which minimises some weighted  $L^\varphi$ average of the distortion. \S \ref{NTFP} deals with the case $\varphi(t)=t$ and $\eta(x)\equiv 1$; minimisers of mean distortion.   

\subsection{Gr\"otzsch type problems} \label{GTP}

The classical Gr\"otzsch problem asks one to identify the linear mapping as the homeomorphism of least maximal distortion between two rectangles.  Put
\[ {\bf Q}_1=[0,\ell]\times [0,1],\hskip20pt {\bf Q}_2=[0,L]\times [0,1] \]
and suppose we have a deformation of finite distortion $ f:{\bf Q}_1\to{\bf Q}_2$  with  
\begin{equation}\label{bc} {\rm Re} f(0,y) = 0, \;\;\; {\rm Re} f(\ell,y) = L, \;\;\; {\rm Im} f(x,0) = 0, \;\;\;  {\rm Im}   f(x,1) = 1 \end{equation}
(so  $f$ is orientation preserving and maps edges to edges).
This Sobolev map is absolutely continuous on lines and so $\int_{0}^{\ell} {\rm Re}( f_x) \;dx = L$ and $\int_{0}^{1} {\rm Im} (  f_y ) \;dy = 1$ for almost all $y$ and $x$ respectively,  and hence
\begin{equation} \label{ak1} 
{\rm Re} \int_{{\bf Q}_1} f_x(z) \;|dz|^2   =  L , \hskip10pt {\rm Im}\int_{{\bf Q}_1} f_y (z) \;|dz|^2   =   \ell.
\end{equation}
The distortion function is
\begin{equation}
\IK(z,f) = \frac{|f_x|^2+|f_y|^2}{J(z,f)} \geqslant 1.
\end{equation}
A Gr\"otzsch problem seeks a minimiser,  satisfying the boundary conditions (\ref{bc}),  to the functional
\begin{equation}
f \mapsto \int_{{\bf Q}_1} \varphi(\IK(z,f)) \; \lambda(z) \; |dz|^2
\end{equation}
for some positive weight function $\lambda$.

\subsection{Equivalence between  Nitsche and  Gr\"otzsch problems}

The universal cover of an annulus is effected by the exponential map,  so $z\mapsto \exp(2\pi z)$ takes $z=x+iy\in [0,L]\times [0,1]$ to ${\mathbb A}_2$ if $\sigma_2=\log(S)=2\pi L$. A branch of logarithm must be chosen to define an ``inverse'' map $[0,\ell]\times[0,1] \to {\mathbb A}_1$.  If  $f:{\mathbb A}_1\to{\mathbb A}_2$ is given,  then we can define
$\tilde{f}(z) = \frac{1}{2 \pi}\log(f(\exp{2\pi z}))$.  A particular point here is that $\log$  is conformal (in fact we only really need $\log$ to define a univalent conformal mapping from 
${\mathbb A}_1\setminus ( [1,S]\times\{0\}) $ to ${\bf Q}_2$ with edges matching up) so
\begin{equation}
\IK(z, \tilde{f}) = \IK \left(z, \frac{1}{2 \pi} \log\big(f(e^{2\pi z})\big)\right) = \IK\big(z,  f(e^{2\pi z})\big), 
\end{equation}
 and hence a change of variables yields
 \begin{eqnarray*}
 \int_{{\bf Q}_1}\varphi\big(\IK(z,\tilde{f})\big) \lambda(z) \; |dz|^2 & = &  \int_{{\bf Q}_1}\varphi\big(\IK(z,  f( e^{2\pi z})\big) \lambda(z) \; |dz|^2 \\
  & = &  4\pi^2\; \int_{{\mathbb A}_1}\varphi\big(\IK(w,  f )\big) \; \lambda(z) e^{-4\pi {\rm Re}(z)} \; |dw|^2 .
 \end{eqnarray*}
With the choice
\begin{equation}\label{etalamda} \eta(w)  = 4\pi^2\; \lambda(z) e^{-4\pi {\rm Re}(z)}, \hskip30pt e^z=w, \end{equation}
 the equivalence between the two problems (with related weight) is seen. Again, the exact branch of log here will be immaterial to our considerations.
 
 \medskip 
 
\begin{remark} In fact the equivalence between  Nitsche and  Gr\"otzsch problems is only when one assumes periodic boundary behaviour in the Gr\"otzsch problem.  We will be fortunate in that  the absolute minimisers for the Gr\"otzsch problem in the situations we consider do exhibit this periodicity and so can be lifted.  
\end{remark}
 
\subsection{Sublinear distortion functionals; $p<1$.}

The purpose of this brief section is to show that  minimisers never exist for the $L^p$-minimisation problem when $p<1$.   We recall from \cite[Theorem 5.3]{AIMO} (actually the proof of this result)
\begin{lemma}\label{lemseq} Let $\Psi(t)$ be a positive strictly increasing function of sublinear growth:
\[\lim_{t\to\infty} \frac{\Psi(t)}{t} = 0 \]
Let ${\mathbb B}={\mathbb D}(z_0,r)$ be a round disk and suppose that $f_0:{\mathbb B}\to {\mathbb C}$ is a homeomorphism of finite distortion with $\int_{\mathbb B} \Psi(\IK(z,f_0)) < \infty$. Then there is a sequence of mappings of finite distortion $f_n:{\mathbb B}\to f_0({\mathbb B})$ with $f_n(\zeta)=f_0(\zeta)$ near $\partial{\mathbb B}$ and with
\begin{itemize}
\item $\IK(z,f_n) \to 1$ uniformly on compact subsets of ${\mathbb B}$
\item $ \int_{\mathbb B} \Psi(\IK(z,f_n)) \to \int_{\mathbb B} \Psi(1)$ as $n\to\infty$.
\end{itemize}
\end{lemma}
We now prove the following theorem.  

\begin{theorem} Let $\Psi(t)$ be a positive strictly increasing function of sublinear growth, 
let $\Omega$ be a  domain and let $\lambda(z)\in L^\infty(\Omega)$ be a positive weight.  Suppose that $g_0:\Omega\to{\mathbb C}$ is a homeomorphism of finite distortion with 
\[ \int_\Omega \Psi(\IK(z,g_0)) < \infty \] 

Then there is a sequence of mappings of finite distortion $g_n:\Omega \to g_0(\Omega)$ with $g_n(\zeta)=g_0(\zeta)$, $\zeta\in \partial \Omega$ with
\begin{equation}
 \int_\Omega \Psi(\IK(z,g_n)) \, \lambda(z)  \to \Psi(1) \int_\Omega \lambda(z) \hskip20pt \mbox{ as $n\to\infty$} 
\end{equation}
\end{theorem}
\noindent{\bf Proof.}  Let $\epsilon > 0$.  Since $\Big( \Psi(\IK(z,g_0))-\Psi(1) \Big) \lambda(z)\in L^1(\Omega)$ we can choose a finite collection of disjoint disks contained in $\Omega$, say $\{B_i\}_{i=1}^{N}$, so that 
\begin{equation}
\Big| \int_{\Omega\setminus\bigcup B_i} \Big( \Psi(\IK(z,g_0))-\Psi(1) \Big) \lambda(z) \, dz \Big| < \epsilon/2
\end{equation}
Next,  for each $i$ we use Lemma \ref{lemseq} in the obvious way to find $h_i:B_i\to{\mathbb C}$ with $h_i=g_0$ in a neighbourhood of $\partial B_i$ and 
\[  \Big|  \int_{B_i} \Psi(\IK(z,h_i)) \lambda(z) - \Psi(1)  \int_{B_i} \lambda(z) \Big|<\frac{\epsilon}{2N} \]
Then the map 
\[ g_\epsilon(z) = \left\{ \begin{array}{ll} g_0(z) & z\in \Omega \setminus \bigcup B_i \\ h_i(z) & z\in B_i \end{array}\right. \]
is of finite distortion and 
\begin{eqnarray*}
\Big|  \int_\Omega \Psi(\IK(z,g_\epsilon)) \, \lambda(z)  -  \Psi(1)  \int_\Omega \lambda(z) \Big| & = & \Big|   \int_{\Omega\setminus\bigcup B_i} \Big( \Psi(\IK(z,g_\epsilon)) -  \Psi(1)\Big)  \, \lambda(z)\, dz \\ && +\sum_{i=1}^{N}    \int_{B_i} \Big( \Psi(\IK(z,h_i)) -  \Psi(1)\Big)  \, \lambda(z)    \, dz \Big| \\
& < & \epsilon
\end{eqnarray*}
The result follows.\hfill $\Box$

\medskip

And the next corollary is what we seek.

\begin{corollary} Let $\Psi(t)$ be a positive strictly increasing function of sublinear growth, 
let $\Omega$ be a   domain and let $\lambda(z)\in L^\infty(\Omega)$ be a positive weight.  Suppose that $g_0:\Omega\to{\mathbb C}$ is a homeomorphism of finite distortion with 
\[  \int_\Omega \Psi(\IK(z,g_0)) < \infty \] 
Then 
\[ \min_{\mathcal F}\,    \int_\Omega \Psi(\IK(z,g)) \, \lambda(z)  = \Psi(1)  \int_\Omega \lambda(z) \, dz \]
with equality achieved by a mapping of finite distortion if and only if the boundary values of $g_0$ are shared by a conformal mapping.  Here ${\mathcal F}$ consists of homeomorphisms of finite distortion $g$ with $g|\partial\Omega=g_0$.  
\end{corollary}

\subsection{Minimisers of convex distortion functionals}

A natural class of homeomorphic mappings between rectangles satisfying the Gr\"otzsch boundary conditions (\ref{bc}) are those of the form
\begin{equation}\label{fform}
f_0(z) = u(x) + iy,
\end{equation}
which will correspond to the lifts of the radial stretchings.
For these mappings we have $(f_0)_x=u_x$ and $(f_0)_y= i$.  We will  show these mappings are the extremals for our mapping problems,  but we will have to deal with degenerate situations as well - in particular where $f_0$ is not well defined,  but has a well defined inverse. These  are  topologically monotone mappings we have talked about earlier as local uniform limits of homeomorphisms,  and so for us as limits of minimising sequences. In order to avoid excess technical complications we make the following assumptions: 

\medskip  Let $w=a+ib\in[0,L]\times[0,1]$ and set
\begin{equation}\label{g0def} g_0(w)= v(a)+ib \end{equation}
 where  $v:[0,L]\to[0,\ell]$ is an absolutely continuous, increasing (but not necessarily strictly increasing) surjection. The derivative of $v_a$ of $v$ is a non-negative $\mathscr L^1([0,\ell])$ function which if it is positive almost everywhere makes $v$ strictly increasing and we may set
\begin{equation}\label{f0def} f_0 = g_{0}^{-1}:[0,\ell] \times [0,1]\to [0,L]\times [0,1] \end{equation}
  
 We now proceed as follows.

\begin{lemma}\label{lemma} Set $g_0(w)= v(a)+ib$,  where $v:[0,L_0]\to[0,\ell]$ is an absolutely continuous, increasing surjection.  Let $f:[0,\ell]\times[0,1]\to [0,L]\times[0,1]$ be a homeomorphism of finite distortion satisfying the boundary conditions (\ref{bc}).  Then
\begin{equation}\label{4.3}
|v_a(a) f_x(g_0(w))+i f_y(g_0(w))|^2 \geqslant 0 
\end{equation}
Equality holds for $f$ and almost every $w$ if and only if $v$ is strictly increasing $L=L_0$ and $f = g_{0}^{-1}$.
\end{lemma}
\noindent{\bf Proof.}  
We consider 
\[ h(w) = (f\circ g_{0})(w)\] 
The mapping $h\in \mathscr W^{1,1}_{loc}$ and maps $[0,L_0]\times[0,1]\to [0,L]\times [0,1]$ respecting the sides.   We compute the $\bar{w}$-derivative of $h$;
\begin{eqnarray*} 2 h_{\bar w}(w)  & = &   f_z(v(a)+ib) v_a(a)+f_{\bar z}(v(a)+ib)v_a(a) \\ && +if_z(v(a)+ib)-if_{\bar z}(v(a)+ib)\\  
& = & f_x(g_0(w))v_a(a) + i f_y(g_0(w))    = 0
\end{eqnarray*}
as an $\mathscr L^1$-function.  Thus $h$ is analytic by the Looman-Menchoff theorem.  The boundary conditions and analyticity imply that $h$ is a homeomorphism of the boundary which must therefore be a homeomorphism of the rectangles.  Then  $L=L_0$ and $h$ must be the identity since the two rectangles have moduli $L_0$ and $L$.    The result  follows \hfill $\Box$

\bigskip 

For a suitable function $v$ as above,  let us write $z=x+iy$ where
\begin{eqnarray}
z = g_0(w) & {\rm and}  &
\omega(x)= v_a(a).
\end{eqnarray} 
We note that $\omega$ is well defined.  First,  $g_0$ is a surjection and if $g_0(w_1)=g_0(w_2)$,  then $w_1$ and $w_2$ lie in a common interval on which $v$ is constant,  whereupon $v_a(a_1)=v_a(a_2)=0$.  However if $\omega(x) > 0$,  then 
\begin{equation}\label{equ}
|\omega(x) f_x(z)+i f_y(z)|^2 \geqslant 0 
\end{equation}
with equality almost everywhere if and only if $g_0$ is a homeomorphism and $f_0=g_{0}^{-1}$.   Also, when $\omega>0$,   $v$ is strictly increasing,  
\begin{eqnarray}\label{fdef}
g_{0}^{-1}(z) = f_0(z) = u(x)+iy, && 
\omega(x) =  {1}/{u_x(x)}.
\end{eqnarray}

\noindent We now suppose that $\omega>0$ and expand  out (\ref{equ}). 
\begin{eqnarray*}
0 & \leqslant & |\omega(x) f_x +i f_y|^2  =  (\omega(x) f_x +i f_y)(\omega(x)\overline{ f_x}-i \overline{f_y}) \\
  & =  & \omega^2(x) |f_x|^2 +|f_y|^2 - 2 {\rm Im}(\omega(x)  f_y  \overline{f_x})  
\end{eqnarray*}
which yields
\begin{equation}\label{3.5}
 \omega^2(x) |f_x|^2 +|f_y|^2 \geqslant  2 \omega(x) \, {\rm Im}(  f_y  \overline{f_x}).  
\end{equation} 

Notice that if we write $f=U+iV$, then
\[ {\rm Im}( f_y  \overline{f_x})   =  {\rm Im} (U_x(z)-iV_x(z))(U_y(z)+iV_y(z)) = J(z,f), \]
so (\ref{3.5}) gives us
\begin{equation}\label{3.6}
 \omega^2(x) |f_x|^2 +|f_y|^2 \geqslant  2 \omega(x)  J(z,f)
\end{equation}
with equality almost everywhere if and only if   $f=f_0$ (with the implication that $f_0$ is a homeomorphism). We can rewrite (\ref{3.6}) in two different ways. Namely
\begin{eqnarray*}
|f_x|^2 +|f_y|^2  & \geqslant & ( 1-\omega^{-2}(x))|f_y|^2  + 2 \omega^{-1}(x)  J(z,f),\\
|f_x|^2 +|f_y|^2 & \geqslant & (1-\omega^2(x))|f_x|^2 + 2 \omega(x)  J(z,f),
\end{eqnarray*}
which gives us two estimates on the distortion function (writing $J=J(z,f)$),
\begin{eqnarray*}
\IK(z,f)  & \geqslant  & ( 1-\omega^{-2}(x)) \frac{|f_y|^2}{J}  + 2 \omega^{-1}(x),  \\
\IK(z,f)  & \geqslant   & (1-\omega^2(x))\frac{|f_x|^2}{J} + 2 \omega(x),
\end{eqnarray*}
Next,  when $\omega>0$ almost everywhere we can define $f_0$ by (\ref{fdef}) with (\ref{uxdef}).  Then 
\begin{eqnarray*} \IK(z,f_0)& = & ( 1-\omega^{-2}(x)) \frac{|(f_{0})_y|^2}{J_0}  + 2 \omega^{-1}(x), \\
\IK(z,f_0)  & = & (1-\omega^2(x))\frac{|(f_0)_x|^2}{J_0} + 2 \omega(x),
\end{eqnarray*}
and thus we have our first useful inequalities
\begin{lemma} If $\omega(x)>0$,  then
\begin{equation}
\IK(z,f) - \IK(z, f_0)  \geqslant   ( 1-\omega^{-2}(x)) \left [ \frac{|f_y|^2}{J} -   \frac{|(f_0)_y|^2}{J_0}   \right],  
\end{equation}
and 
\begin{equation}
\IK(z,f) - \IK(z, f_0)    \geqslant   (1-\omega^2(x)) \left[ \frac{|f_x|^2}{J} -   \frac{|(f_0)_x|^2}{J_0}   \right]  \end{equation}
with equality holding almost everywhere in either inequality if and only if $f=f_0$.
\end{lemma}
 We leave it to the reader to establish the elementary inequality for complex numbers $X$, $X_0$ and real $J$, $J_0$,
\begin{equation}\label{ineq}
\frac{|X|^2}{J} -\frac{|X_0|^2}{J_0} \geqslant  2 \, {\rm Re} \Big( \frac{\overline{X_0}}{J_0} \; \big(X-X_0\big) \Big) - \frac{|X_0|^2}{J_{0}^{2}} \big(J-J_0\big),
\end{equation}
with equality holding if and only if $X/X_0=J/J_0$ is a positive real number (expand $|X/X_0-J/J_0|^2\geqslant 0$). This shows
\[ (X,Y,J) \mapsto \frac{|X|^2+|Y|^2}{ J} \]
to be convex on ${\mathbb C}\times{\mathbb C}\times{\mathbb R}_+$.
We  apply (\ref{ineq}) and this requires that the coefficient $( 1-\omega^{-2}(x))  >0$ in the first case or $(1-\omega^2(x))>0$ in the second.  Since this depends on $u_x$ for the candidate extremal mapping,  we carry along the two inequalities and write $\IK_0 = \IK(z, f_0)$.  First note that if $\varphi:{\mathbb R}\to{\mathbb R}$ is convex,  then its graph lies above any tangent line: \[ \varphi(\IK)-\varphi(\IK_0) \geqslant \varphi'(\IK_0)\big(\IK-\IK_0). \]
Notice that if $\varphi'' >0$, equality quality holds here if and only if $\IK=\IK_0$.
This therefore yields the following two inequalities:
\begin{eqnarray*}
\varphi(\IK(z,f)) - \varphi(\IK(z, f_0))  & \geqslant &  ( 1-\omega^{-2}(x))\varphi'(\IK_0) \\ && \left [  2 \, {\rm Re} \Big( \frac{\overline{(f_0)_y}}{J_0}\;\big(f_y-(f_0)_y\big)\Big)  - \frac{|(f_0)_y|^2}{J_{0}^{2}} \big(J-J_0\big)   \right],
\\
 \\
\varphi(\IK(z,f)) -\varphi( \IK(z, f_0))  &  \geqslant  & (1-\omega^2(x)) \varphi'(\IK_0)  \\ && \left [  2 \, {\rm Re} \Big( \frac{\overline{(f_0)_x}}{J_0}\; \big(f_x-(f_0)_x\big) \Big) - \frac{|(f_0)_x|^2}{J_{0}^{2}} \big(J-J_0\big) \right].
\end{eqnarray*}

Now $(f_0)_y=i$ and $(f_0)_x=1/\omega(x)= J_0$ so these equations read as
\begin{eqnarray}
\lefteqn{\varphi(\IK(z,f)) - \varphi(\IK(z, f_0)) }\nonumber  \\ & \geqslant &  \big( 1-\frac{1}{\omega^2(x)}\big)\varphi'(\IK_0)  \left [  \frac{2}{J_0} \, {\rm Im} \big( f_y-1\big) - \frac{J-J_0}{J_{0}^{2}} \right]  \nonumber 
\\
& = & 2  \left(  \omega(x)-\frac{1}{\omega(x)} \right)\varphi'(\IK_0)  {\rm Im} \big( f_y-1\big)  \nonumber \\  && +  \big( \omega^2(x)-1 \big)\varphi'(\IK_0)  (J_0-J )\,,  \\
 \nonumber  \\
\label{li}\lefteqn{ \varphi(\IK(z,f)) -\varphi( \IK(z, f_0))  }  \nonumber  \\  &  \geqslant  & (1-\omega^2(x)) \varphi'(\IK_0)    \left [  2 \, {\rm Re} \big( f_x-(f_0)_x \big) -  \big(J-J_0\big) \right].
\end{eqnarray}
 
 Now we want to multiply these two inequalities by a weight function $\lambda(x)$ and integrate. 
We  are naturally led to consider the Euler-Lagrange equation for the variational problem minimising 
\[ \int_{\bf Q} \varphi(\IK(z,f))\lambda(x)\;|dz|^2 \]
 among functions of the form (\ref{fform}). This equation reduces to the next equation in one real variable
\begin{equation} \frac{d}{dx} \Big[ \lambda(x) \Big(1-\frac{1}{u_{x}^{2}}\Big)  \varphi'\Big(u_x+\frac{1}{u_x}\Big) \Big] =  0  \end{equation}
We would therefore like  $\omega(x)$  to be chosen chosen so that
\begin{equation}\label{omegachoice}
\lambda(x) (1-\omega^2(x))  \varphi'\Big(\omega(x)+\frac{1}{\omega(x)}\Big) = \alpha \neq 0
\end{equation}
for a real constant $\alpha$.  This equation implicitly defines $\omega$ directly,  and does not involve any of its derivatives.  

\medskip
\begin{remark}  We postpone a discussion of boundary values for the solution $f_0$ (really $g_0$) that we seek.   Set 
\begin{equation}\label{Lbv} \int_{0}^{\ell} \frac{dx}{\omega(x)} = L_0 \end{equation}
The boundary conditions we want are that $L=L_0$ to identify the minimum.  However,  if $L_0<L$,  then Lemma \ref{lemma} still applies - and we obtain strict inequality.  Also,  we note that from (\ref{omegachoice}), with an assumption that $\lambda>0$ and $\varphi'$ are continuous, that $\omega=0$ implies that $\lambda(x)   \varphi'(\infty) = \alpha$.  In particular,  we cannot have $\omega(x)=0$ unless $\varphi'$ is bounded - a condition we will see again.   
\end{remark} 
\medskip

We
  now suppose that we have (\ref{omegachoice}) holding almost everywhere and $L_0<L$. Then (\ref{omegachoice}) forces $0\leq \omega(x) <1$ for all $x$ or $\omega(x)>1$ for all $x$.  The case $\omega\equiv 1$,  $\alpha=0$ yielding $g_0=f_0=identity$.  The first case (where we will ultimately have to deal with degeneration as we cannot guarantee the boundary conditions) has $u_x>1$ and so must correspond to stretching $L>\ell$.  In the other case $\ell<L$.

We proceed as follows.
\begin{eqnarray*}
\int_{{\bf Q}_1}   \varphi(\IK(z,f)) \lambda(x) \; |dz|^2 & \geqslant & \int_{{\bf Q}_1}\varphi( \IK(z, f_0) ) \lambda(x)\; |dz|^2   -\alpha  \int_{{\bf Q}_1} ( J_0-J ) \; |dz|^2 \\ + &  2  & \int_{{\bf Q}_1}\lambda(x) \left(  \omega(x)-\frac{1}{\omega(x)} \right)\varphi'(\IK_0)  {\rm Im} \big( f_y-1\big) \; |dz|^2 , \\
\int_{{\bf Q}_1}   \varphi(\IK(z,f)) \lambda(x)\; |dz|^2 & \geqslant & \int_{{\bf Q}_1}\varphi( \IK(z, f_0) ) \lambda(x)\; |dz|^2  + \alpha  \int_{{\bf Q}_1}  ( J_0-J ) \; |dz|^2 \\ & +  & 2 \alpha \int_{{\bf Q}_1} \, {\rm Re} \big( f_x-(f_0)_x \big) \; |dz|^2 .
\end{eqnarray*}
For an arbitrary Sobolev homeomorphism it is well known that
\[ \int_{{\bf Q}_1} J\;|dz|^2 \leqslant |{\bf Q}_2| = L= \int_{0}^{\ell} u_x (x)\, dx = \int_{{\bf Q}_1} J_0\;|dz|^2  \]
We will use the first inequality above when $\alpha <0$ and the second when $\alpha>0$.  Thus, for $\alpha<0$
\begin{eqnarray*}
\lefteqn{ \int_{{\bf Q}_1}   \varphi(\IK(z,f)) \lambda(x) \; |dz|^2 } \\ & \geqslant & \int_{{\bf Q}_1}\varphi( \IK(z, f_0) ) \lambda(x)\; |dz|^2   +   2    \int_{{\bf Q}_1}\lambda \left(  \omega -\frac{1}{\omega } \right)\varphi'(\IK_0)  {\rm Im} \big( f_y-1\big) \; |dz|^2 ,  
\end{eqnarray*}
while for $\alpha>0$ we have
\begin{eqnarray*}
\lefteqn{ \int_{{\bf Q}_1}   \varphi(\IK(z,f)) \lambda(x)\; |dz|^2 } \\ & \geqslant & \int_{{\bf Q}_1}\varphi( \IK(z, f_0) ) \lambda(x)\; |dz|^2  +   2 \alpha \int_{{\bf Q}_1} \, {\rm Re} \big( f_x-(f_0)_x \big) \; |dz|^2 
\end{eqnarray*}
Next,  from (\ref{ak1}) we see that 
\begin{eqnarray*}
\lefteqn{\int_{{\bf Q}_1}\lambda(x) \left(  \omega(x)-\frac{1}{\omega(x)} \right)\varphi'(\IK_0)  {\rm Im} \big( f_y-1\big) \;|dz|^2}   \\ & = &  \int_{0}^{\ell}    \lambda(x) \left(  \omega(x)-\frac{1}{\omega(x)} \right)\varphi'(\IK_0)  \Big[ \int_{0}^{1}  {\rm Im} \big(  f_y-1 \big) dy \Big]\; dx = 0
\end{eqnarray*}
and 
\[   \int_{{\bf Q}_1} {\rm Re} \big( f_x-(f_0)_x  \big)\;|dz|^2 =  \int_{0}^{1} \Big[  \int_{0}^{\ell} {\rm Re} \big( f_x-(f_0)_x  \big)  dx \Big]\; dy  = 0. \]
 Thus we have established 
\begin{theorem}\label{mainthm} Let $\lambda(x) >0$ be a positive weight  and $\varphi:[1,\infty)\to[0,\infty)$ be convex increasing. Let the function  $u:[0,\ell]\to[0,L]$
\begin{equation}
u(0)=0, \hskip30pt u(\ell)=L_0 \leq L
\end{equation}
be a solution to the ordinary differential equation
 \begin{equation}\label{ode}
\lambda(x) \left(1-\frac{1}{u_{x}^{2}(x)}\right)  \varphi'\left(u_x(x)+\frac{1}{u_x(x)} \right) = \alpha 
\end{equation}
where $\alpha$ is a nonzero constant.  Set
\begin{equation}
f_0(z) = u(x)+i y, \hskip20pt f_0:[0,\ell]\times[0,1] \to [0,L_0]\times [0,1].
\end{equation}
Let   $ f:[0,\ell]\times[0,1] \to [0,L]\times [0,1]$ be a  surjective homeomorphism of finite distortion  with 
\[ {\rm Re} f(0,y) = 0, \;\;\; {\rm Re} f(\ell,y) = L, \;\;\; {\rm Im} f(x,0) = 0, \;\;\;  {\rm Im} f(x,1) = 1 .\]
Then
\begin{equation}
 \int_{{\bf Q}_1}   \varphi(\IK(z,f)) \lambda(x)  \;|dz|^2 \geqslant   \int_{{\bf Q}_1}\varphi( \IK(z, f_0) ) \lambda(x)  \;|dz|^2.
 \end{equation}
 Equality holds if and only if $f=f_0$.  In particular,  if  $L_0<L$,  then this inequality is strict. 
\end{theorem}
Notice $\alpha=0$ gives the identity mapping - clearly always an absolute minimiser when it is a candidate.

\subsection{Degenerate Cases}  
 
 Theorem \ref{mainthm} identifies the extremal homeomorphism of finite distortion when we can find $\alpha$ so that $L_0=L$.  We will see later that this is not always possible and then  Theorem \ref{mainthm} provides us with the unattainable lower bound $\int_{{\bf Q}_1}\varphi( \IK(z, f_0) ) \lambda(x) $ - since the inequality is strict.  When $L_0<L$ of course $f_0$ is not a candidate mapping for the minimisation problem - so it might not be surprising the bound is unattainable.  However it might be possible that this {\em value} is the limit of a minimising sequence of candidates. What we want to do here is to find circumstances in which this happens.
 
 \begin{theorem}  Suppose that $f_0$ is defined as in Theorem \ref{mainthm} and for no choice of $\alpha$ is it possible that $L=L_0$.  Suppose that $\varphi'$ is bounded.  Then there is a sequence of surjective homeomorphism of finite distortion $f_j:[0,\ell]\times[0,1]\to[0,L]\times [0,1]$ such that 
 \begin{equation}
  \int_{{\bf Q}_1}   \varphi(\IK(z,f_j)) \lambda(x)  \;|dz|^2 =  \int_{{\bf Q}_1}   \varphi(\IK(z,f_0)) \lambda(x)  \;|dz|^2
 \end{equation}
 In particular,  under these circumstances there is no extremal homeomorphism of finite distortion whatsoever. 
 \end{theorem}
 \noindent{\bf Remark.} We will see in the next few sections the condition $\phi'$ bounded is necessary for nonexistence of minimisers,  but not sufficient.  The behaviour of the weight $\lambda$ near its minimum determines whether we can solve the boundary problem for arbitrary $L$.
 
 \medskip
 
 \noindent{\bf Proof.}  Our assumption is that $\varphi$ is convex increasing and thus $\varphi'$ is positive and increasing,   not necessarily strictly.  We may also assume $\lim_{t\to\infty}\varphi'(t)=1$. The function $t\mapsto  (1- {t}^{-2})\varphi' (t+1/t)$   is strictly increasing and our solution $u^{\alpha}$ is obtained by the rule $u_{x}^{\alpha}(x)=t_x$ where $(1- t_{x}^{-2})\varphi' (t_x+1/t_x)=\alpha/\lambda(x)$.  This implies that $\alpha\leq \alpha_0 =\min_{x} \lambda(x)$. It is easy to see that $u^\alpha(\ell) \nearrow u^{\alpha_0}(\ell)$ and our hypothesis is that last value is $L_0<L$.  Thus for $\alpha\leq \alpha_0$, the family $u^\alpha\in \mathscr W^{1,1}([0,\ell])$, with a uniform bound.  Further $u_0 = u^{\alpha_0}$ is strictly increasing with derivative tending to $\infty$ as $x$ approaches a minimum,  say $x_0$, of $\lambda$ (which may be an endpoint of $[0,\ell]$).  Let 
 \[ g_0(w) = v_0(a)+ib,  \hskip20pt v = u_{0}^{-1} \]
Then $(v_0)_a(a)=1/(u_0)_x(x)$ with $u_0(x)=a\in [0,L_0]$.  With $u_0(x_0)=a_0$ we have $(v_0)_a(a_0)=0$.  We now define a new function $g:[0,L]\times[0,1]\to[0,\ell]\times[0,1]$ by simply defining $g(w) = v(a)+ib$ to be constant near $x_0$.  That is (with appropriate modification should $x_0$,  the minimum of $\lambda$  be an endpoint)
\begin{equation}
v(a) = \left\{\begin{array}{ll} v_0(a) & a \leq a_0 \\  v_0(a_0) & a_0 \leq a \leq  a_0+L-L_0 \\  v_0(a+L_0-L) &  a_0+L-L_0 \leq a \leq   L 
\end{array}
\right.
\end{equation}
 Then $v_a$ is a non-negative $L^1$ function,  vanishing on $[a_0,a_0+L-L_0]$ and with $\|v_a\|_{1} = \ell$.  It is routine to approximate $v_a$ by positive $v^{j}_{a}$ in $L^1$ and with $\|v_{a}^{j}\|_{1} = \ell$.  Define $v(a)=\int_{0}^{a} v^{j}_{a}$ to get a homeomorphic mapping of finite distortion $g^j(w)=v^j(a)+ib$.  Notice that $g^j\to g$ uniformly in $\mathscr W^{1,1}([0,L]\times[0,1])$.  Set 
 \[ f^j = (g^j)^{-1}:[0,\ell]\times[0,1]\to[0,L]\times[0,1] \]
 The mappings $f^j$ are surjective homeomorphisms of finite distortion.  We calculate,  with the change of variables $g^j(w)=z$,
 \begin{eqnarray*}
 \int_{{\bf Q}_1} \varphi(\IK(z, f^j)) \lambda(z)\, dz & = &  \int_{{\bf Q}_2} \varphi\Big(\frac{\|Df^j(g^j)\|^2}{J(g^j,f^j)}\Big)\, \lambda(g^j(w)) \, J(w, g^j) dw \\
 & = &  \int_{{\bf Q}_2} \varphi\Big( \|(Dg^j(w))^{-1}\|^2 J(w,g_j) \Big)\, \lambda(g^j(w)) \, J(w, g^j) dw \\
  & = &  \int_{{\bf Q}_2} \varphi\Big(  v_{a}^{j}(a) +\frac{1}{v_{a}^{j}(a)}\Big)\, \lambda(v^j(a)) \, v_{a}^{j}(a) da \\
    & \to  &  \int_{{\bf Q}_2} \varphi\Big(  v_{a}(a) +\frac{1}{v_{a}(a)}\Big)\, \lambda(v(a)) \, v_{a}(a) da \\
    & =  &  \int_{[0,L_0]\times[0,1]} \varphi\Big(  (v_0)_{a}(a) +\frac{1}{(v_0)_{a}(a)}\Big)\, \lambda((v_0)(a)) \, (v_0)_{a}(a) da \\
        & =  &  \int_{[0,\ell]\times[0,1]} \varphi\Big(  (u_0)_{x}(x) +\frac{1}{(u_0)_{x}(x)}\Big)\, \lambda(x) \,  dx \\
        & =  &  \int_{{\bf Q}_1} \varphi\Big(  \IK(z, f_0) \Big)\, \lambda(z) \,  dz  
\end{eqnarray*}
 
\section{The Nitsche phenomenon}

Before moving on to discuss the theory in more generality we provide a couple of interesting applications based around the classical Nitsche problem.

\medskip

Theorem \ref{mainthm} strongly motivates us to study the ordinary differential equation (\ref{ode}) for solutions will identify minima of out Nitsche and Gr\"otzsch type problems.  Note also that the transformation from the Nitsche type problem to the Gr\"otzsch problem yields a significantly simpler equation to study---in fact it's not really an ODE at all. 

\subsection{Weighted mean distortion}\label{wmd1}
Let us first observe how the Nitsche phenomenon arises,  here we have (ignoring multiplicative constants) $\lambda(x)=e^{4\pi x}$ as
$\eta(w)  = 1$.  We have
\[
 1-\frac{1}{u_{x}^{2}(x)}   =\alpha e^{-4\pi x}, \hskip30pt 
u_{x}(x)  =\frac{1}{\sqrt{1-\alpha e^{-4\pi x}}} \]
\[ u(x)   = \int \; \frac{e^{2\pi x}\;dx}{\sqrt{e^{4\pi x}-\alpha }}  = \frac{1}{2\pi} \int \; \frac{dt}{\sqrt{t^2-\alpha }}, \hskip20pt t=e^{2\pi x}. \]
So
\begin{eqnarray*}
u(x) & = & \frac{1}{2\pi} \log \Big(\frac{e^{2\pi x} +\sqrt{e^{4\pi x} - \alpha}}{ 1+\sqrt{1-\alpha}}\Big), \hskip30pt \alpha \neq  0 \\ 
\end{eqnarray*}
noting $u(0)=0$. Recall $u:[0,\ell]\to [0,L]$ and we must solve $u(\ell)=L$,  that is
\begin{equation}\label{fdef} L = \frac{1}{2\pi} \log \Big(\frac{e^{2\pi \ell} +\sqrt{e^{4\pi \ell} - \alpha}}{ 1+\sqrt{1-\alpha}}\Big) \end{equation}
by choice of our free parameter $\alpha$.  Notice that $\alpha$ is not bounded from below, and as $\alpha \to -\infty$ we can make the right hand side of (\ref{fdef}) as small as we like. Thus there is always a minimiser if $L\leqslant \ell$.  If $\alpha >0$ we see that (\ref{ode}) requires $\alpha < 1$ so that
\[  L <  \frac{1}{2\pi} \log \Big(\frac{e^{2\pi \ell} +\sqrt{e^{4\pi \ell} - 1}}{ 1+\sqrt{1-\alpha}}\Big)    \]
and when unwound,  these are precisely the Nitsche bounds.  

\medskip

For more general weights $\lambda(x)$, 
\[
 1-\frac{1}{u_{x}^{2}(x)}   =\frac{\alpha}{\lambda(x)}, \hskip30pt 
u_{x}(x)  =\sqrt{\frac{\lambda(x)}{\lambda(x)-\alpha }} \]
and we must typically study the behaviour of an integral like
\[ u(x)   = \int_{0}^{\ell} \; \sqrt{\frac{\lambda(x)}{\lambda(x)-\alpha }}\; dx . \]
Again,  as $\alpha\to -\infty$ and if $\lambda$ is not too bad,  we can make this integral as small as we wish. Notice that $\alpha/\lambda(x) < 1$, so if we put $\lambda_0=\min_{x\in [0,\ell]} \lambda(x)$,  then this integral is dominated by the one with the choice $\alpha=\lambda_0$ and the issue is to decide whether
\[  \int_{0}^{\ell} \; \sqrt{\frac{\lambda(x)}{\lambda(x)-\lambda_0 }}\; dx < \infty. \]
If this integral is finite,  then we will observe Nitsche type phenomena; non-existence of minima outside a range of moduli.  

Supposing that $\lambda_0>0$,   the principal issue concerns the integral
\begin{equation}  \int_{0}^{\ell} \; \frac{dx}{ \sqrt{\lambda(x)-\lambda_0 }}  < \infty, \end{equation}
and without going into excessively fine details,  convergence will require that
\[ \lambda(t) \approx \lambda_0 + t^{2s},  \hskip30pt s<1 \]
near the minimum.

In particular,  if $\lambda$ is a smooth positive weight and $\lambda'(x)=0$ at it's minimum (which may well occur at the endpoints in which case we choose the appropriate left or right derivative),  then we can always solve the deformation problem.

\subsection{$\varphi' $ unbounded; $p>1$.} We show that if the convex function $\varphi$ has unbounded derivative,  then there is always a minimiser,  with mild assumptions on the weight function $\lambda$.  In particular we do not see the Nitsche phenomenon for the $L^p$--norms of mean distortion.  First observe that when $\varphi$ is smooth and convex increasing,  the function
\[ F(t) =  \left(1-\frac{1}{t^{2} }\right)  \varphi'\left(t+\frac{1}{t} \right) \]
is increasing for $t>0$,  indeed
\[ F'(t) =  \frac{2}{t^3}  \varphi'\left(t+\frac{1}{t} \right) + \left(1-\frac{1}{t^{2} }\right)^2  \varphi''\left(t+\frac{1}{t} \right) > 0 \]
Next, if $\varphi'$ is unbounded,  it is monotone and then 
\[ \lim_{t\searrow 0} F(t) = -\infty,  \hskip20pt  \lim_{t\to +\infty} F(t) =+\infty \]
The intermediate value theorem implies that for each $x\in (0,\ell)$ and $\alpha \in {\mathbb R}$ we can find $t_x>0$ so that
$F(t_x) =  {\alpha}/{\lambda(x)}$.
We then define a function $v_\alpha$ by the rule
$v_\alpha(x) = t_x > 0$.
Then $v$ is a positive function which certainly satisfies 
\begin{equation} 
 \lambda(x) \left(1-\frac{1}{v^{2}(x)}\right)  \varphi'\left(v(x)+\frac{1}{v(x)} \right) = \alpha 
\end{equation}
The regularity of the function $v_\alpha$ depends on that of $\lambda$.  The function $u$ that we are looking for define the mapping $f$ is an antiderivative of $v$.  For $f$ to be a mapping of finite distortion,  we'll need that $u$ is absolutely continuous.  These conditions are all easily seem to be true if $\lambda$ (and hence $v_\alpha$) is piecewise continuous.   

We then define
\begin{equation}
u_x(x) = v_\alpha(x)
\end{equation}
If $\lambda$ is bounded and bounded away from $0$,  then it is easy to see that $v_\alpha$ is uniformly large when $\alpha$ is chosen large,  while $v_\alpha$ is uniformly small if $\alpha$ is chosen large and negative.  Further  

\[ u(x) = \int_{0}^{x} v_\alpha (s) \; ds\]
depends continuously on $\alpha$ (as $v_\alpha$ depends piecewise continuously).  Thus $u(\ell)$ can be made to assume any positive value - in particular we can solve $u(\ell)=L$,  and so we don't see the Nitsche phenomena.  Here is a theorem summarising this discussion. The reader will see that we have not striven for maximum generality.

\begin{theorem}\label{mainthm} Let $\lambda(x)$ be a piecewise continuous positive weight  bounded and bounded away from $0$.  Let $\varphi:[1,\infty)\to[0,\infty)$ be smooth and convex increasing with $\varphi'(s)$ unbounded as $s\to\infty$. Then the minimisation problem
\begin{equation}
\min_{f\in {\mathcal F}}  \int_{{\bf Q}_1}   \varphi(\IK(z,f)) \lambda(x)  \;|dz|^2  
 \end{equation}
 has a unique solution of the form $f(z)=u(x)+iy$.  Here ${\mathcal F}$ is the family of all mappings of finite distortion satisfying the boundary conditions described in \ref{GTP}
\end{theorem}

We then have the following corollary about the weighted $L^p$-norms of distortion functions.

\begin{corollary} Let $\lambda(x)$ be a piecewise continuous positive weight  bounded and bounded away from $0$.  Then the minimisation problem
\begin{equation}
\min_{f\in {\mathcal F}}  \int_{{\bf Q}_1}   \IK^p (z,f)  \lambda(x)  \;|dz|^2  
 \end{equation}
has a unique solution of the form $f(z)=u(x)+iy$.  Here ${\mathcal F}$ is the family of all mappings of finite distortion satisfying the boundary conditions described above.
\end{corollary}

\subsection{Critical   case: $\varphi'$ bounded.}

Examining the above argument we see that in this case we can always find a solution to the minimisation problem of the given form if $L<\ell$ by varying $\alpha$ among negative values, $\alpha=0$ produces the identity mapping.   However,   there are further subtleties. The reader will quickly get to a condition on the integrability of $ \psi(\lambda_0/\lambda(x))$ where $\psi$ is the inverse of the bounded increasing  function $t\mapsto \varphi'(t+t^{-1})(1-t^{-2})$ with $\lambda_0=\min_{[0,\ell]} \lambda$.  Let us give two illustrative examples in the standard (Nitsche) case with $\ell=1$, $\lambda(x)=e^{-4\pi x}$.  We may assume that $\varphi'(t)\nearrow 1$  and the limiting case $\alpha=e^{4\pi}$:

\bigskip

{\bf Case:} $\varphi(t)=t-\log(t)$,  $\varphi'(t) = 1-\frac{1}{t}$,  $a=a(x)=e^{4\pi (x-1)} \leqslant 1 $.

\medskip

We choose $u_x$ to be the largest real root of the polynomial:
\begin{eqnarray*}
 \Big(1-\frac{1}{t+t^{-1}}\Big)\big(1-\frac{1}{t^2}\big) & = & a \\
p(t) = -1+t-at^2-t^3+(1-a)t^4 & = & 0.
 \end{eqnarray*}
Since 
\[ p(\frac{1}{1-a})=-1+\frac{1}{1-a} - \frac{a}{(1-a)^2}-\frac{1}{(1-a)^3}+\frac{1}{(1-a)^3}  = - \frac{a^2}{(1-a)^2} < 0 \]
the largest real root 
$ u_x(x)  >  {1}/({1-a(x)})$
 and 
 \[ \int_{0}^{x} u_y(y)  >  \int_{0}^{x}  \frac{1}{1-e^{4\pi (y-1)}}  \approx  \frac{1}{4\pi}  \log\Big(\frac{1}{1-x}\Big) \] 
 and this diverges as $x\to 1$.  Therefore with appropriate choice of $\alpha$ we can always solve $u(0)=0$ and $u(1)=L$. Hence there is no Nitsche phenomena.
 
\bigskip

{\bf Case: $\varphi(t)=t+\frac{1}{(p-1)t^{p-1}}, p>0,p\neq 1$}.

\medskip

We have $\varphi'(t) = 1-\frac{1}{t^p}$,  $0<a=a(x)=e^{-4\pi x} < 1 $ for $0 < x <1$, and hence $u_x$ is the largest real root of the polynomial
\begin{equation}\label{eq:polyp}
P(t)  =  \left(1-\frac{1}{(t+t^{-1})^p}\right) \left(1-\frac{1}{t^2}\right) -a = 0.
\end{equation}
Note that when $t > 0$, $P(t)$ is a continuous monotonically increasing function of $t$. Also note that $P(1) = -a < 0 $, and
$
\lim_{t \to \infty} P(t) = 1-a > 0,
$
so that $P$ has exactly one real positive root $u_x  > 1$. 

First let us deal with $0<p<1$. Observe that
\[
\big(1-(1-a)^2\big)\big(\big(1+(1-a)^2\big) -
(1-a)\big) - a \big(1+(1-a)^2\big) = -a^2(1-a)^2 < 0.
\]
This may be rewritten as
\[
\Big(1- \Big(\frac{1}{1-a}\Big)^{-2}\Big)\Big(1 -
\frac{1}{\frac{1}{1-a}+\frac{1-a}{1}}\Big) - a \;<\; 0
\]
Now using the fact that $0 < p < 1$, we see that
\[
P\Big(\frac{1}{1-a}\Big) =
\Big(1-\frac{1}{\Big(\frac{1}{1-a}\Big)^2}\Big)\Big(1 -
\frac{1}{\Big(\frac{1}{1-a}+\frac{1-a}{1}\Big)^p}\Big) - a < 0
\]
and hence the largest real root
$
u_x >  {1}/{(1-a)}.
$
The integral of the right hand side diverges (see the reasoning for the case
$\varphi^\prime = 1-t^{-1}$). Thus with appropriate choice for
$\alpha$ we can always solve $u(0) = 0, u(1) = L$ and therefore we see no
Nitsche phenomena for $p < 1$.

Next, take $p \geqslant 2$. Recall (\ref{eq:polyp}). Note that $\left(t+\frac{1}{t}\right)^p > \left(t+\frac{1}{t}\right)^2 > t^2$. Set $Q(t)$ as
\[
P(t) = \left(1-\frac{1}{(t+t^{-1})^p}\right) \left(1-\frac{1}{t^2}\right) -a >  \left(1-\frac{1}{t^2}\right)^2 -a = Q(t), \quad t>1.
\]
The largest real root of $P(t)$ is therefore dominated by the largest real root of $Q(t)$. Solving $Q(t) = 0$ gives 
\[
\int_0^1 u_x\;dx < \int_0^1 \frac{1}{\sqrt{1-e^{-2\pi x}}}\; dx = \log \left(e^{\pi} + \sqrt{e^{2 \pi} - 1}\right),
\]
a finite number. Therefore, when $p \geqslant 2$, $u_x(x)$ is dominated by an integrable function and we must see the Nitsche phenomenon. It is no coincidence that the value of the integral here is strongly reminiscent of that for the Nitsche case (\ref{fdef}); the integrands for that case and the estimate here are very similar.

It remains to cover the case where $1 < p < 2$. Note that for $p > 1$,
$
1-\frac{1}{(t+t^{-1})^p} > 1-\frac{1}{t^p},
$
and for $p < 2$,
$
1-\frac{1}{t^2} > 1-\frac{1}{t^p}.
$
Therefore the polynomial 
\[
P(t)  =  \left(1-\frac{1}{(t+t^{-1})^p}\right) \left(1-\frac{1}{t^2}\right) - a > \left(1-\frac{1}{t^p}\right)^2 - a = Q(t),
\]
and the largest real root of $P(t)$ is again dominated by the largest real root of $Q(t)$. Solving $Q(t) = 0$ yields
$
u_x <  \left(1-\sqrt{a(x)}\right)^{-1/p}.
$
Near $x = 0$, $\sqrt{a(x)} = e^{-2 \pi x} \approx 1-2\pi x $ and so
\[
\int_0^1 \frac{1}{\left(1-\sqrt{a(x)}\right)^{1/p}}\; dx \approx \left(\frac{1}{2 \pi}\right)^{1/p} \int_0^1 \frac{1}{x^{1/p}}\; dx,
\]
which converges if and only if $p > 1$. Therefore in this case, too, $u_x$ is dominated by an integrable function and we must see a critical Nitsche-type phenomenon.

\section{Teichm\"uller's Problem.}  In the preceding (frictionless)  examples we have seen the case of minimisers of $L^p$-mean distortion that $p=1$ is rather special and that $p\in(1,\infty]$ are similar in that extremal mappings exist and are regular. We just mention here a related problem with boundary values in which the exact opposite occurs.  Namely $p\in [1,\infty)$ have the same nature (nonexistence of minimisers,  see \cite{GJM} for the case $p=1$.) while $p=\infty$ has minimisers.  Teichm\"uller's problem for mean distortion is to identify for $r>0$,
\[ \inf \quad \|\IK(z,f)\|_{L^p(\DD)} \]
and show a minimiser exists.  Here the infimum is taken over all mappings $f :\DD\onto \DD$ of finite distortion with $f\in \mathscr W^{1,2}_{loc}(\DD)$,  $f(0)=r$,  and that can be extended to a homeomorphism of the closed disk  onto itself with $f|\partial \DD = Id$. The classical Teichm\"uller  problem is $p=\infty$ where the maximal distortion is employed instead of the mean distortion and minimisers exist and are of Teichm\"uller type.  That is $\mu=k\bar{\phi}/{\phi}$ where $\phi$ is meromorphic with a pole of order $1$ at $r$.  However,  for $1\leq p<\infty$ minimisers exist in a weak sense,  and there is an associated Ahlfor-Hopf meromorphic quadratic differential with a pole of order $1$,  but these minimisers can never be locally quasiconformal  except in the trivial case $r=0$ and $f(z)=z$,  \cite{MY}.


\begin{thebibliography}{99}
   
\bibitem{AIMb}
K. Astala, T. Iwaniec, and G. Martin, \textit{Elliptic partial differential equations and quasiconformal mappings in the plane}, Princeton University Press, 2009.

\bibitem{AIM}
K. Astala, T. Iwaniec, and G. Martin, \textit{Deformations of annuli with smallest mean distortion}, Arch. Ration. Mech. Anal. {\bf 195} (2010), no. 3, 899--921.

\bibitem{AIMO}
K. Astala, T. Iwaniec, G. J.  Martin, and J. Onninen, J. \textit{
Extremal mappings of finite distortion.} Proc. London Math. Soc. (3)
{\bf 91} (2005), no. 3, 655--702.

 \bibitem{Bac}
J. M. Ball, \textit{Convexity conditions and existence theorems in nonlinear elasticity},  Arch. Rational Mech. Anal. {\bf 63} (1976/77), no. 4, 337--403.

\bibitem{Ba0}
J. M. Ball, \textit{Global invertibility of Sobolev functions and the interpenetration of matter}, Proc. Roy. Soc. Edinburgh Sect. A, {\bf 88} (1981), 315--328.
 
\bibitem{Ba1}
J. M. Ball, \textit{Discontinuous equilibrium solutions and cavitation in nonlinear elasticity}, Philos. Trans. R. Soc. Lond. A {\bf 306} (1982) 557--611.
 
 \bibitem{BallCurrieOlver}
J. M. Ball,  J. C. Currie, and P. J. Olver, \textit{Null Lagrangians,
weak continuity, and variational problems of arbitrary order}, J.
Funct. Anal. {\bf 41} (1981), no. 2, 135--174.
  
\bibitem{ForgotenCaccioppoli}
R. Caccioppoli, \textit{Funzioni pseudo-analitiche e rappresentazioni pseudo-conformi delle superfice riemanniane}, Ricerche Mat. {\bf 2} (1953) 104--127.
 
 \bibitem{Ch} G. Choquet, {\em Sur un type de transformation analytique
g\'{e}n\'{e}ralisant la repr\'{e}sentation conforme et d\'{e}finie
au moyen de fonctions harmoniques,} Bull. Sci. Math., {\bf 69},
(1945), 156-165.

\bibitem{CL}
S. Conti and C. De Lellis, \textit{Some remarks on the theory of elasticity for compressible Neohookean materials}, Ann. Sc. Norm. Super. Pisa Cl. Sci.  
 
\bibitem{deFranchis}  M. de  Franchis, {\em La piti generale funzione d’invarianza per criteri sufficienti di minim0 con condizioni di Dirichlet per integrali pluridimensionali del primo ordine dipendenti da un vettore a piu componenti}, Atti Accad. Naz. Lincei Rend. Cl. Sci. Fis. Mat. Natur., {\bf 37} (8) (1964), 129--140.

\bibitem{Edelen}
 D. G. B.  Edelen, \textit{The null set of the Euler-Lagrange operator},
Arch. Rational Mech. Anal. {\bf 11} (1962) 117--121.

\bibitem{EL}  J. Eells and  L. Lemaire, {\em  A Report on Harmonic Maps}, Bulletin London Math. Soc., {\bf 10}, (1978), 1--68.

\bibitem{EL2}
J. Eells and L. Lemaire, \textit{Another report on harmonic maps}, Bulletin London Math. Soc.,  \textbf{20},  (1988),   385--524.

\bibitem{ES} J. Eells and J.Y. Sampson, {\em Harmonic Mappings of Riemannian Manifolds}, American Journal of Mathematics, {\bf 86}, (1964),  109--160.
 
\bibitem{Ev}
L. C. Evans, \textit{Quasiconvexity and partial regularity in the calculus of variations}, Arch. Rational Mech. Anal. {\bf 95} (1986), no. 3, 227--252.

\bibitem{Ericksen} J.L. Ericksen, {\em Nilpotent energies in liquid crystal theory},  Arch. Rat. Mech. Anal., {\bf 10}, (1962), 189--96
 
 \bibitem{GrecoIwaniecSubramanian}
L. Greco, T. Iwaniec, and U. Subramanian, \textit{Another approach to biting convergence of Jacobians}, Illinois Journal of Mathematics, vol. 47, No. 5 (2003), pp. 815--830.
 
\bibitem{HKb}
S.  Hencl, and  P. Koskela, \textit{Lectures on mappings of finite distortion}, Lecture Notes in Mathematics, 2096. Springer, Cham, (2014).

\bibitem{HL2} F. Hang and F. Lin, \textit{Topology of Sobolev mappings II}, Acta Math. {\bf 191} (2003), 55--107.

 \bibitem{HK} S. Hencl and  P. Koskela, {\em Lectures on mappings of finite distortion}, Lecture Notes in Mathematics, {\bf 2096},  Springer, Cham, 2014. xii+176 pp
 
 \bibitem{HP}
S. Hencl and A. Pratelli  \textit{Diffeomorphic Approximation of $\,\mathscr W^{1,1}\,$ -Planar Sobolev Homeomorphisms},  J. Eur. Math. Soc. to appear.
 
 \bibitem{Ho}
H. Hopf,  \textit{Differential geometry in the large}, Notes taken by Peter Lax and John Gray. With a preface by S. S. Chern. Lecture Notes in Mathematics, 1000. Springer-Verlag, Berlin, 1983.
   
   \bibitem{IL-arma}
T. Iwaniec, A. Lutoborski, \textit{Integral Estimates for Null Lagrangians}, Arch. Rational Mech. Anal. {\bf 125} (1993), 25--79.
 
\bibitem{IKOni}
T. Iwaniec, L. V. Kovalev, and J. Onninen, \textit{The Nitsche conjecture}, {J. Amer. Math. Soc. \textbf{24} (2011), no.~2, 345--373.}
 
 \bibitem{IMb}
T. Iwaniec and G. Martin,  \textit{Geometric Function Theory and Non-linear Analysis}, Oxford Mathematical Monographs, Oxford University Press, (2001).

 \bibitem{IMO2}  T. Iwaniec,   G.J. Martin and J. Onninen {\em On Minimisers of $L^p$-mean Distortion},  
Computational Methods and Function Theory,  {\bf 14}, (2014),  399--416.

\bibitem{IwaniecOnninenH1}
T. Iwaniec  and J. Onninen,  \textit{$\,\mathcal H^1\,$-estimates of Jacobians by subdeterminants}, Mathematische Annalen  \textbf{324}, (2002), 341--358.
  
  \bibitem{IOan}
T. Iwaniec  and J. Onninen,  \textit{$n$-Harmonic mappings between annuli} Mem. Amer. Math. Soc. {\bf 218} (2012).

  \bibitem{IOp}
T. Iwaniec  and J. Onninen, 
{\em Mappings of smallest mean distortion and free-Lagrangians}, 
Ann. Sc. Norm. Super. Pisa Cl. Sci., {\bf 20},  (2020), 1--106. 
 
\bibitem{IOPR} T. Iwaniec, J. Onninen, P. Pankka, and T. Radice, arXiv:2004.03381v1.

\bibitem{Kn}
H. Kneser, {\em L\"{o}sung der Aufgabe 41,} Jahresber. Deutsch.
Math.-Verein., {\bf 35}, (1926), 123-124.
 
 \bibitem{Landers}
 J.W. Landers, {\em Invariant multiple integrals in the Calculus of Variations Contributions to the Calculus
of Variations}, (1938-1941) (Chicago: University of Chicago Press) pp 175-208.

 \bibitem{GJM} G.J. Martin, {\em The Teichmm\"uller problem for mean distortion}, Ann. Acad. Sci. Fenn. Math, {\bf  34},  (2009),  , 233--247. 
 
   
   \bibitem{MM} G.J. Martin and M. McKubre-Jordens,  
{\em  Deformations with smallest weighted $L^p$ average distortion and Nitsche-type phenomena},  
J. Lond. Math. Soc., {\bf  85},  (2012),  282--300. 

 \bibitem{MY} G.J. Martin and C. Yao {\em The Teichm\"uller problem for $L^p$-means of distortion}, Math. arXiv:2107.07660 (2021) to appear.
 
\bibitem{Morrey}
C. B.  Morrey, \textit{Quasi-convexity and the lower semicontinuity of multiple integrals},  Pacific J. Math. {\bf 2}, (1952). 25--53.

\bibitem{MullerQiYan}
S. M\"uller, S.  Qi, T. and Yan, B.S., \textit{On a new class of elastic deformations not allowing for cavitation},
Ann. Inst. H. Poincar\'{e} Anal.. Non Lineaire {\bf 11} (2)  (1994), 217--97243.
 
 \bibitem{Nash}
J. Nash,  \textit{$C^1$ isometric imbeddings},  Ann. Math. (2) 60, (1954) 383--396.
 
 \bibitem{Nitsche}
J.C.C. Nitsche,   \textit{On the modulus of doubly connected regions under harmonic mappings},  Amer. Math. Monthly,  {\bf 69}, (1962), 781--782.
 
 \bibitem{Rad}
T. Rad\'{o}, \textit{Aufgabe 41.}, Jahresber. Deutsch. Math.-Verein.,
{\bf 35}, (1926), 49.
 
 \bibitem{Rund} H. Rund, {\em TheHamilton-Jacobi theory in the Calculus of Variations}, London:vanNostrand, 1966.
 
  \bibitem{Zhang}
K. Zhang, \textit{Biting theorems for the Jacobians and their applications},  Ann. Inst. Henri Poincar\'{e} {\bf 7} no. 4, (1990), 345--365.

\end{thebibliography}
 \end{document}